\documentclass{amsart}
\usepackage[procnames]{listings}
\usepackage{cite}
\usepackage{color}
\usepackage{ amssymb }
\usepackage{hyperref}
\usepackage{amsmath}
\usepackage{amsmath}
\usepackage{mathtools}
\usepackage{amssymb}
\numberwithin{equation}{section}
\title{The Siegel variance formula for quadratic forms}
\author{Naser T. Sardari}
\address{Department of Mathematics, UW-Madison, Madison, WI 53706}

\email{ntalebiz@math.wisc.edu}

\date{\today}

	\newtheorem{thm}{Theorem}[section]
	
	\newtheorem{prop}[thm]{Proposition}
	
	\newtheorem{rem}[thm]{Remark}	
	\newtheorem{lem}[thm]{Lemma}
	\newtheorem{conj}[thm]{Conjecture}
	\renewcommand{\vec}[1]{\mathbf{#1}}

	\newtheorem{cor}[thm]{Corollary}

		\def\va{\text{Var}}

\begin{document}
\maketitle
\begin{abstract}
We introduce a smooth variance sum associated to a pair of positive definite symmetric integral matrices $A_{m\times m}$ and $B_{n\times n}$, where $m\geq n$. By using the oscillator representation,  we give a formula for this variance sum in terms of a smooth sum over the square of a functional evaluated on the $B$-th Fourier coefficients of the vector valued holomorphic  Siegel modular forms which are Hecke eigenforms and obtained by the theta transfer from $O_{A_{m\times m}}$.  By using the Ramanujan bound on the Fourier coefficients of the holomorphic cusp  forms, we give a sharp upper bound on this variance when $n=1$. As applications,  we  prove a  cutoff phenomenon for the probability that a unimodular lattice of dimension $m$  represents  a given  even number. This gives an optimal upper bound on the sphere packing density of almost all even unimodular lattices. Furthermore, we  generalize  the result of  Bourgain, Rudnick and Sarnak~\cite{Bourgain}, and also give an optimal bound on the diophantine exponent of the $p$-integral points on any positive definite $d$-dimensional quadric, where $d\geq 3$. This improves the best known bounds due to  Ghosh, Gorodnik and Nevo~\cite{GGN} into an optimal bound.     
%
%
\end{abstract}

\tableofcontents 

\section{Introduction}

\subsection{Statement of results} In this section, we discuss two applications  of the Siegel variance formula (Theorem~\ref{Siegelthmm} for $n=1$).
 \subsubsection{The cutoff phenomenon  in large dimension} Suppose that $A_{m\times m}$ is a  positive definite symmetric integral matrix with determinant 1, and $C(A)$ denotes the genus of $A$ which is a finite set.  It is well known that $C(A)$ has only two possibilities, namely even or odd unimodular lattices. There is a natural probability measure defined by Siegel \cite{Siegel1} on $C(A)$:
\[
\mu_{s}(A_i):=\dfrac{ \frac{1}{|O_{A_i}(\mathbb{Z})|}}{\sum_{A_i\in C(A)} \frac{1}{|O_{A_i}(\mathbb{Z})|}},
\]   
where $s\in\{0,1\}$ depending on $C(A)$ being even or odd, and  $|O_{A_i}(\mathbb{Z})|$ is the size of the integral orthogonal group of $A_i.$
   The first application is on bounding the probability that an odd integer $q$ (even number $2q$)  is representable by an odd unimodular lattice (even unimodular lattice) of dimension $m$ with respect to $\mu_{s}$. Every integer (even integer) is representable over $p$-adic integers $\mathbb{Z}_p$ by an odd unimodular lattice (even unimodular lattice) of dimension $m\geq 4.$ This fact  and an application of the Hardy-Littlewood circle method  implies that every large enough integer (even integer) with respect to $m$  is representable by every odd unimodular (every even unimodular lattice).  This is a version of our theorem which shows a  cutoff phenomenon at point  $q\sim \frac{m}{2\pi e}$ ($2q\sim \frac{m}{2\pi e}$) for the probability measure $\mu_{s}$.
 \begin{thm}\label{spheredense} Let $m$ be even and  $q$ be an odd integer.  We have 
 
 \[
\mu_{1}\big(\vec{x}^{\intercal}A\vec{x}=q \text{ for some }  \vec{x} \in \mathbb{Z}^m \big)=\begin{cases} \leq 5113^{-t } & \text {  for } q\leq \frac{m}{2\pi e}+\frac{0.5}{2\pi e} \log(m)-t-1,  \\  \geq 1- 5133^{- t }   & \text {  for } q\geq \frac{m}{2\pi e}+\frac{1.6}{\pi e}\log(m) +t,   \end{cases}
\]
where $0 \leq t=o(m).$ Similarly for $\mu_0,$ we have 
\[
\mu_{0}\big(\vec{x}^{\intercal}A\vec{x}=2q \text{ for some }  \vec{x} \in \mathbb{Z}^m \big)=\begin{cases} \leq 5113^{-t } & \text {  for } 2q\leq \frac{m}{2\pi e}+\frac{0.5}{\pi e}\log(m) -t-1,  \\  \geq 1- 5113^{-t }   & \text {  for } 2q\geq \frac{m}{2\pi e}+\frac{2.6}{\pi e}\log(m) +t,   \end{cases}
\]
where $0 \leq t=o(m).$
Note that $5113=\lfloor e^{\pi e}\rfloor.$
 \end{thm}
 We give the proof of Theorem~\ref{spheredense} in Section~\ref{sec2}. We have the following conjecture. 
 \begin{conj}\label{conj}
 Let $q\geq (1+\epsilon) \frac{m}{2\pi e}$ for some fixed $\epsilon>0.$ Then $q$ is representable by every odd unimodular lattice of dimension $m\gg_{\epsilon} 1.$ Similarly, every even integer $2q\geq (1+\epsilon) \frac{m}{2\pi e}$ is representable by every even unimodular lattice of dimension $m\gg_{\epsilon} 1.$
 \end{conj}
\begin{rem}
Let $\delta>0$ and $L$ be an even unimodular lattice of dimension $m$. Theorem~\ref{spheredense} implies the sphere packing density of $L$ is less than $m^{2+\delta}2^{-m}$ with $\mu_0$-probability $1+O(m^{-\delta+\epsilon})$ for any $\epsilon>0.$ Moreover, if  $A$ is an odd uniomodular lattices, then the sphere packing density  is less than $m^{1+\delta}2^{-m}$ with $\mu_1$-probability $1-m^{-\delta+\epsilon}$ for any $\epsilon>0.$ 
 The problem of studying unimodular lattices with large sphere packing densities has studied by several authors; see \cite{Conwayb}.  The best known upper bounds on the sphere packing density of unimodular lattices  is $(1.424)^m 2^{-m}$  \cite{Sloane}, while the best known lower bound for the density of lattices (not necessarily integral)  is $m\log\log(m)2^{-m}$ \cite{Akshay}. So, there is an exponential gap between the upper bound and the lower bound for the sphere packing density of  unimodular lattices. We substantially  improve the upper bound and show that the sphere packing density is $o(m^{2+\epsilon}2^{-m})$ for all but a tiny fraction of unimodular lattices with respect to the Siegel mass probability. Conjecture~\ref{conj} implies the sphere packing density of even unimodular lattices is less than $(1+\epsilon)^m2^{-m}.$
\end{rem}
\subsubsection{Optimal equidistribution of the integral points on quadrics}The second application  is on the distribution of the integral points on quadrics. Suppose that  $F(x_1,\dots,x_m)$ is a  positive definite integral quadratic form in $m\geq 3$ variables with discriminant $D$. Let $N>0$ be an integer where $\gcd(N,2D)=1$, and define
$$V_N(R):=\big\{(x_1,\dots,x_m): F(x_1,\dots,x_m)=N, \text{ and } x_i\in R \text{ for } 1 \leq i \leq m \big\},$$
where $R$ is any commutative ring.  Assume that $V_N(\mathbb{Z}_p)\neq \emptyset$ for every prime $p.$
 Let $O_{F}$ be the orthogonal group associate to the quadratic form $F(x_1,\dots,x_m).$ Note that $V_1(\mathbb{R})$ is a compact  homogenous variety  with  the action of $O_{F}(\mathbb{R}).$
Let $\mu$ be the unique $O_{F}(\mathbb{R})$ invariant probability measure defined on $V_1(\mathbb{R}).$ Suppose that $k(x)$ is a fixed positive smooth function  supported  on $(-2,2),$ and $k(x)=1$ for $x\in (-1,1).$
 Let 
\[K_{\eta}(\vec{x},\vec{y}):=C_{\eta}k\big(\frac{\sqrt{F(\vec{x}-\vec{y})}}{\eta}\big),\] where $\vec{x},\vec{y}\in V_1(\mathbb{R}),$ $\eta\in \mathbb{R}$ and $C_{\eta}$ is a normalization factor such that
\[
\int_{ V_1(\mathbb{R})} K_{\eta}(\vec{x},\vec{y}) d \mu(\vec{y})=1.
\]
 We note that $K_{\eta}(\vec{x},\vec{y})$ is a point-pair invariant function, which means 
 \[K_{\eta}(g\vec{x},g\vec{y})=K_{\eta}(\vec{x},\vec{y})\] for every $g\in O_{F}(\mathbb{R}).$
  Let  $R_F(N):= |V_N(\mathbb{Z})|.$ For $m\geq 4$ and from $\gcd(N,2D)=1$~\cite[Remark 1.7]{Optimal}, it follows that
\begin{equation}\label{lohardy}
N^{m/2-1-\epsilon} \ll R_F(N) \ll  N^{m/2-1+\epsilon}.
\end{equation}
For $m=3,$ we further assume that $N\neq t_i \mathbb{Z}^2$ for finitely many $\{ t_i \}$ that defines the exceptional-type square classes; see \cite{Hanke}. Then by Siegel's ineffective bound $L(1,\chi_q)> q^{-\epsilon},$ we have the same bounds as in \eqref{lohardy}. 
   Define
\begin{equation}\label{varsum}
\va_F(N,\eta):= \int_{V_1(\mathbb{R}) } \big(K_{\eta}(\vec{x},N)- R_F(N)\big)^2 d\mu(\vec{x}),
\end{equation}
where $K_{\eta}(\vec{x},N):=\sum_{\vec{y}\in  \frac{1}{\sqrt{N}} V_N(\mathbb{Z})} K_{\eta}(\vec{x},\vec{y}).$ The following theorem is an application of our main results.
\begin{thm}\label{mthm} Let $F$ and $N$ be as above. For $m=3,$ suppose that  $N\neq t_i \square$, where $\{ t_i \}$ defines the finitely many  exceptional-type square classes. 
  Assume   either of these assumptions 
  \begin{itemize} 
 \item $m$ is even, 
 \item$N=sl^2$ for some bounded square free integer $s$, 
\item Lindel\"of hypothesis holds for the holomorphic modular forms.
\end{itemize} Then,  we have
\[
\va_F(N,\eta)\ll \frac{N^{\epsilon} R_F(N)} {\eta^{m-1}},
\]  
where the implied constant in $ \ll$ only depends on $F$ and $\epsilon >0.$
\end{thm}
We give the proof of Theorem~\ref{mthm} in Section~\ref{sec3}.  When $m=3$ Theorem~\ref{mthm} essentially follows from the work of Bourgain, Rudnick and Sarnak~\cite{Bourgain}.  They verified, with respect to different statistical tests, that the distribution of the integral points on the 2-sphere is similar to the distribution of a Poisson process. However, for $m\geq 4$ it was observed by Wright~\cite{Wright,Wright2} as mentioned in~\cite{Sarnak2} that they are big regions on $V_1(\mathbb{R})$ which repels the integral points. 
 Let $ C(\vec{x},\eta):=B(\vec{x},\eta)\cap V_1(\mathbb{R})$ be a cap of radius $\eta>0$ centered at $\vec{x}\in V_1(\mathbb{R}),$ where $B(\vec{x},\eta)$ is the Euclidean ball of radius $\eta>0.$  Wright~\cite{Wright,Wright2} showed that there are caps of size $\eta \gg N^{-\frac{1}{4}}$ which does not intersect  $\frac{1}{\sqrt{N}}V_N(\mathbb{Z}).$ In~\cite{Optimal}, we proved that every cap of size $\eta \gg N^{-\frac{1}{4}+\delta}$ contains an integral point for $m\geq 5$ and every $\delta>0.$  The following corollary implies that on average the covering properties of $\frac{1}{\sqrt{N}}V_N(\mathbb{Z})$ for every $m\geq4$ is optimal and is as good as the Poisson process. 
\begin{cor}\label{sthm}
 Assume  the conditions of Theorem~\ref{mthm}. Let  $\eta \gg R_F(N)^{-\frac{1}{(m-1)}+\epsilon}\sim N^{-\frac{m-2}{2(m-1)}+\epsilon}$, then all but a tiny fraction of the caps of $V_1(\mathbb{R})$ with the radius $\eta$ intersects  $\frac{1}{\sqrt{N}}V_N(\mathbb{Z}).$ On the other hand if $\eta \ll R_F(N)^{-\frac{1}{(m-1)}-\epsilon} $ then only a tiny fraction of them intersect $\frac{1}{\sqrt{N}}V_N(\mathbb{Z}).$
\end{cor}
We give the proof of the Corollary~\ref{sthm} in Section~\ref{sec3}. See also the work of Ellenberg,  Michel and Venkatesh~\cite{Ellenberg} for the non-archimedean version  of the above corollary for $F=x_1^2+x_2^2+x_3^2$ under the Linnik condition on $N$. When $N=p^{2k}$ for some fixed prime number $p$ and $F=x_1+\dots +x_d^2$ for $d=3,4$ the above corollary follows from the work of Ghosh, Gorodnik and Nevo\cite{GGN}. 

 We discuss two related applications of Theorem~\ref{mthm} in what follows. First,  we recall  the  definition of  the averaged  covering exponent of the integral points on the sphere.   Let $S^{m-1}(\mathbb{R})$
be the sphere of radius $1$ in $\mathbb{R}^m.$   Let $S^{m-1}_N(\mathbb{Z})$ be the set of integral points 
\[
S^{m-1}_N(\mathbb{Z}):=\big\{ (x_1,\dots,x_m)\in \mathbb{Z}^{m}: x_1^2+\dots+x_m^2=N   \big\},
\]
where $0 \leq N\in \mathbb{Z}.$ We have $\frac{1}{\sqrt{N}}S^{m-1}_N(\mathbb{Z}) \subset S^{m-1}(\mathbb{R}).$  
Let $N_{\delta,\epsilon}$ denote the minimum integer  such that all but $\epsilon^{\delta}$ fraction of caps $C(\vec{x},\epsilon)$ of size $\epsilon$ on $S^{m-1}(\mathbb{R})$ contain a point of $\frac{1}{\sqrt{N}}S^{m-1}_N(\mathbb{Z}).$ 
%
%
Sarnak defined  \cite{Sarnak2}
 the averaged covering exponent of the  integral points on the sphere by:
 \begin{equation*}
 \begin{split}
 \bar{K}_{m}&:=\lim_{\delta\to 0}\limsup_{\epsilon \to 0}\frac{\log \big(\#S^{m-1}_{N_{\delta,\epsilon}}(\mathbb{Z}) \big) }{\log\big( 1/\text{vol }(C(\vec{x},\epsilon))\big)}.
 \end{split}
 \end{equation*}
 By the Pigeonhole principle, it is easy to see that $\bar{K}_{m}\geq 1.$ Sarnak  proved that $K_4=1$ \cite{Sarnak2}.  This implies  the optimal covering properties of the golden quantum gates inside $SU(2)$; see \cite{Parzanchevski}, \cite{Sardaric}. Sarnak's method is based on the spectral theory of modular forms and uses the Ramanujan bound on the Fourier coefficient of the modular forms. It relies on the coincides that $S^3$ is isomorphic to the units of quaternions. In particular,  the analogues result for $S^{m-1}$  does not follow. The following  is a corollary of Theorem~\ref{mthm}. 
\begin{cor}\label{sggn}
 Assume that $m\geq 4$ is even. Then, 
 $$
\bar{K}_{m}=1.
 $$
 \end{cor}
  For $\vec{x}\in S^{m-1}(\mathbb{Q}),$ let $H(\vec{x}):= \prod_{q} \max_{1\leq i\leq m}(1,|x_i|_q),$ where $|.|_q$ is the $q$ adic valuation of $x_i.$  Fix a prime $p\equiv 1 \mod 4.$ Then $S^{m-1}(\mathbb{Z}[1/p])$ is dense in $S^{m-1}(\mathbb{R}).$  We  prove the following quantitative  form of the diophantine properties of $S^{m-1}(\mathbb{Z}[1/p]).$
\begin{cor}\label{strongGGN} Let $m\geq 3.$
For almost every $x\in S^{m-1}(\mathbb{R}),$ $\delta>0$, and $\varepsilon \in (0,\varepsilon_0(x,\delta)),$ there exists $z \in  S^{m-1}(\mathbb{Z}[1/p]) $ such that 
\[
|x-z|_{\infty}\leq \epsilon \text{ and } H(z)\leq \epsilon^{-\frac{m-1}{m-2}-\delta}.
\]
We note this exponent is the best possible.
\end{cor}

The  above corollary  answers  a question of Ghosh, Gorodnik and Nevo; see \cite{GGN, GGN1,GGN2}. By using the best bound on  the generalized Ramanujan conjecture,  they proved the above corollary~\cite[Page 12]{GGN} for $m=3,4$ and the following exponents for $m\geq5$

\[ H(z)\leq \epsilon^{-2-\delta} \text{ for even } m, \text{ and } H(z)\leq \epsilon^{-\frac{2(m-1)}{m+2}-\delta} \text{ for odd } m.\]
 They raised the question of improving  these bounds in \cite[Page 11]{GGN}. As pointed out above and in the abstract, we give a definite answer to this question. We give the proof of  Corollary~\ref{sggn} and~\ref{strongGGN} in Section~\ref{sec3}.
\subsection{The Siegel variance formula}
In this section we discuss our method. We  introduce a variance sum associated to a pair of positive definite symmetric integral matrices and a smooth compactly supported function. By using  the oscillator representation, we obtain a formula for this variance sum in terms of the Fourier coefficients of the homomorphic  Siegel modular forms which are Hecke eigenform. We denote this formula by the Siegel variance formula; see \eqref{mainform}. We apply this formula to prove Theorem~\ref{spheredense} and  Theorem~\ref{mthm}. More generally, this can be used to study  the distribution of the integral solutions of the representation of a quadratic form by another one. 

Let $A $ and $B$ be two positive definite symmetric  integral matrices with dimensions  $m$ and $n,$ respectively.
Let 
\(C(A):=\{A_1,\dots,A_h\}\)  
be a representative set for the genus class of $A.$  Let 
\[V_{A_i,B}(R):=\{\vec{X}\in M_{m\times n}(R): \vec{X}^{\intercal}A_i \vec{X}=B    \},
\]
where $R$ is a commutative ring. We say $V_{A_i,B}(R)$ is the set of $R$ points of the representation variety of $B$ by $A_i.$

Next, we associate a variance sum associated to $V_{A_i,B}(\mathbb{Z})\subset V_{A_i,B}(\mathbb{R})$ for each $A_i\in C(A).$ The variance sum  only depends on a fixed smooth bump function of size $r$ defined on $\mathbb{R}^n$, and it is independent of the choice of the representative $A_i$ in its equivalence class. Note that $O_{A_i}(\mathbb{R})$ acts on $V_{A_i,B}(\mathbb{R})$ by matrix multiplication, and this action is transitive for $m\geq n$.  We begin by defining a point-pair $O_{A_i}(\mathbb{R})$ invariant function on the representation variety $V_{A_i,B}(\mathbb{R}).$  
Suppose that $k:\mathbb{R}^n\to \mathbb{R}$ is a fixed positive smooth function with compact support.  Define 
\(
|\vec{x}|_i:=\sqrt{ \vec{x}^{\intercal}A_i \vec{x}}
\)  for $\vec{x}\in \mathbb{R}^m$. Let 
\begin{equation}\label{pp}K_{r,B}(\vec{X},\vec{Y}):= C_{r,B}k\big(\frac{|\vec{x}_1-\vec{y}_1|_i}{r}, \dots,\frac{|\vec{x}_n-\vec{y}_n|_i}{r}\big),
\end{equation}
 where $\vec{X},\vec{Y}\in V_{A_i,B}(\mathbb{R}),$   $\vec{x}_j$ and $\vec{y}_j$ are the $j$-th column of $X$ and $Y$ respectively for $1\leq j\leq n,$ and $C_{r,B}$ is a  constant where 
\begin{equation}\label{normalization}
\int_{V_{A_i,B}(\mathbb{R})} K_{r,B}(\vec{X},\vec{Y})d\mu_i(\vec{X})=\int_{V_{A_i,B}(\mathbb{R})} K_{r,B}(\vec{X},\vec{Y})d\mu_i(\vec{Y})=1,
\end{equation}
where $d\mu_i$  is invariant by the action of $O_{A_i}(\mathbb{R})$ and it is  normalized such that  $\int_{V_{A_i,B}(\mathbb{R})} d\mu_i(\vec{Y})=1.$ Note that $ K_{r,B}(\vec{X},\vec{Y})$ is a point-pair invariant function
\[
K_{r,B}(g\vec{X},g\vec{Y})=K_{r,B}(\vec{X},\vec{Y}),
\]
where $g\in O_{A_i}(\mathbb{R})$ and $g\vec{X}$ is the matrix multiplication. This implies $C_{r,B}$ is independent of $\vec{X}\in V_{A_i,B}(\mathbb{R}).$ 
Finally, we define the following variance sum associated to $A_i$, $B$ and $r$ 
\begin{equation}\label{varind}
\va(A_i,B,r):=\int_{V_{A_i,B}(\mathbb{R})} \Big( \big(\sum_{\vec{Y}\in  V_{A_i,B}(\mathbb{Z}) }K_{r,B}(\vec{X},\vec{Y})\big)- R_{A_i}(B)  \Big)^2  d\mu_i(\vec{X}),
\end{equation}
where $R_{A_i}(B):=|V_{A_i,B}(\mathbb{Z})|.$ 
  We define the Siegel variance of representing  $B$ by the genus class of $A$ at scale $r$ by:
\begin{equation}\label{Siegelvarfor}
\va(B,r):=\dfrac{\sum_{A_i} \frac{1}{|O_{A_i}(\mathbb{Z})|}\Big(\va(A_i,B,r)+\big(R_{A_i}(B)-R(B)\big)^2   \Big)}{\sum_{A_i\in C(A)} \frac{1}{|O_{A_i}(\mathbb{Z})|}},
\end{equation}
where 
\begin{equation}\label{averep}
R(B):=\dfrac{\sum_{A_i\in C(A)} \frac{1}{|O_{A_i}(\mathbb{Z})|} R_{A_i}(B)}{\sum_{A_i\in C(A)} \frac{1}{|O_{A_i}(\mathbb{Z})|}},
\end{equation}
is the weighted number of the integral representation of $B$ by the genus class of $A_i.$
 By the Siegel mass formula \cite{Siegel}, we have 
 \begin{equation}\label{Siegelmf}R(B)=\sigma_{\infty}(A,B)\prod_p \sigma_p(A,B),
 \end{equation} where 
 \[
 \sigma_p(A,B):=\lim_{k\to \infty}\frac{|\{X\in V_{A_i,B}(\mathbb{Z}/p^k\mathbb{Z}   \}|}{p^{k(mn-n(n+1)/2)}},
 \]  
 and 
 \[
 \sigma_{\infty}(A,B):=\alpha(m,n)|A|^{\frac{-n}{2}}|B|^{\frac{m-n-1}{2}},
 \]
 where $\alpha(m,n)$ is a fixed constant which depends  only on $m$, $n$; see \cite{Siegel1}.
 Note that $\va(B,r)$ measures how uniform the integral points of the representation varieties of different genus classes are distributed among balls  of size $r$.

 Before stating our main result, we introduce some notations from the theory of automorphic forms and the oscillator representation.  
 We give the detailed descriptions of them in Section~\ref{spec} and Section~\ref{SWF}.  Let $\mathbb{A}_{\mathbb{Q}}=\mathbb{R}\times \hat{\prod}_p^{\mathbb{Z}_p} \mathbb{Q}_p$ be the ring of adeles which is the restrictive direct  product of $\mathbb{R}$ and $\mathbb{Q}_p$ with respect to $\mathbb{Z}_p.$  Fix $\vec{E}\in V_{A,I}(\mathbb{R})$ and the lattice $(\mathbb{Z}^m,A).$ There exists $\sqrt{B}\in M_{n\times n}(\mathbb{R})$ such that $\sqrt{B}^{\intercal}  \sqrt{B}=B.$  
We also fix a choice of $\sqrt{B}$ for every positive definite symmetric matrix $B.$ We note that $\vec{E}_B:=\vec{E}\sqrt{B} \in V_{A,B}(\mathbb{R}).$ Let $O_{\vec{E},A}(\mathbb{R})\times O_A(\prod_p \mathbb{Z}_p)$ be the stabilizer of $(\vec{E},\mathbb{Z}^m)$ by the action of $O_{A}(\mathbb{R})\times O_A(\prod_p \mathbb{Q}_p),$ which is the same as the stabilizer of $(\vec{E}_B,\mathbb{Z}^m)$. This gives the following isomorphism: 
 \[
 \bigcup_{A_i\in C(A)} O_{A_i}(\mathbb{Z})\backslash V_{A_i,B}(\mathbb{R}) = O_A(\mathbb{Q})\backslash O_A(\mathbb{A}_{\mathbb{Q}})/ O_{\vec{E},A}(\mathbb{R})O_A(\prod_p \mathbb{Z}_p) .
 \]
We write  the following spectral decomposition
\begin{equation}\label{spectraldec}
L^2\big(  O_A(\mathbb{Q})\backslash O_A(\mathbb{A}_{\mathbb{Q}})/ O_{\vec{E},A}(\mathbb{R})O_A(\prod_p \mathbb{Z}_p)\big)= \bigoplus_{\pi} \bigoplus_{j=1}^{d_\pi} \phi_{\pi,j},
\end{equation}
where $ \bigoplus_{j=1}^{d_\pi}  \phi_{\pi,j}$ is a finite sum over an orthonormal basis  of $O_{\vec{E},A}(\mathbb{R})O_A(\prod_p \mathbb{Z}_p)$ invariant  harmonic polynomials which generate an irreducible automorphic representation isomorphic $\pi$ of   $L^2\left(O_A(\mathbb{Q})\backslash O_A(\mathbb{A}_{\mathbb{Q}})\right)$ ($\pi$ may appear with multiplicity). More explicitly, the restriction of $ \phi_{\pi,j}$ to $V_{A_i,B}(\mathbb{R})$  is a harmonic polynomial with respect to $A_i.$ Moreover, $ \phi_{\pi,j}$ generates an irreducible representation $ \tau_{\pi_{\infty}} \times \pi_{\infty}=\tau(\lambda)\times \lambda$ by the action of $GL_n(\mathbb{C})\times O_A(\mathbb{C})$ on $M_{m\times n}(\mathbb{R})$, where $\tau(\lambda)\times \lambda$ is a finite dimensional irreducible representation of $GL_n(\mathbb{C})\times O_A(\mathbb{C})$ acting on $\mathcal{H}({\lambda})$ which is  the $\lambda$ isotropic subspace of harmonic polynomials; see Section~\ref{sphesec} and  \cite[Section 2.5.39]{Vergne}.  By using the result of Kashiwara and Vergne~ \cite[Section 6]{Kashiwara},  one can describe the explicit parameters of $\tau(\lambda)\times \lambda$ in terms of the highest weight vectors.  

 In Section~\ref{SWF}, by using the  oscillator representation and fixing an appropriate Siegel theta kernel, we associate a holomorphic Siegel modular form $\Theta(\phi_{\pi,j})(Z)$ with values in the vector space $\mathcal{H}(\lambda)^*$ (dual vector space of $\mathcal{H}(\lambda)$).  In Proposition~\ref{ppww}, we describe explicitly the weight  and the level of the associated Siegel modular form. 
%
%
%
In Proposition~\ref{Heckeosc}, we show that  $\Theta(\phi_{\pi,j})(Z)$ is  an eignform of the Hecke operators defined on the space of Siegel modular forms. We also express  the Fourier coefficient of the associated Siegel modular forms in terms of the Weyl sums of the automorphic forms on the orthogonal group; see Theorem~\ref{weylfourier}.

 Let $h_{r,B}(\pi_{\infty})$ be the  spherical transformation of the point-pair invariant function $K_{r,B}$ at $ \pi_{\infty}$, see equation~\eqref{sphtrans}. Let $\Theta(\phi_{\pi,j},B)\in \mathcal{H}(\lambda)^{*}$ be the $B$-th Fourier coefficient of $\Theta(\phi_{\pi,j}).$ Furthermore, we  define a harmonic polynomial $p_{\lambda,\vec{E}} \in \mathcal{H}(\lambda);$ see Section~\ref{harmgen}. Finally, we state our main theorem. 
 \begin{thm}\label{Siegelthmm}
 We have
 \begin{equation}\label{mainform}
 \va(B,r)=\sum_{\pi} |h_{r}(\pi_{\infty})|^2 \sum_{j=1}^{d_{\pi}}    \left|\langle {\tau_{\pi_{\infty}}(\sqrt{B})^{\intercal}}^{-1}\Theta(\phi_{\pi,j},B), p_{\lambda,\vec{E}}  \rangle \right|^2  .
 \end{equation}
 \end{thm}
\begin{rem}
Suppose that $n=1$ and $B=N \in \mathbb{Z}^{+}$. In the Siegel variance formula, the automorphic forms $\phi_{\pi,j}$ that are associated to the degree $k$ harmonic polynomials (the total dimension is $k^{m-2}$)  that contribute to the variance are the one which are the theta lift from weight $k$ holomorphic modular forms defined on $SL_2$ (the dimension grows linearly in $k$). By comparing the dimension of them and using Howe one-to-one correspondence, it follows that $\Theta(\phi_{\pi,j},B)=0$ for all $\phi_{\pi,j}$  unless $\phi_{\pi,j}$  comes from a lift of $SL_2$ weight $k$ modular form. This and the Ramanujan bound  $|N^{-k/2}\Theta(\phi_{\pi,j},N)|^2\ll N^{\frac{m}{2}-1} $ are the source of the equidistribution of the integral points at the optimal scale. 
\end{rem}

%
%
%
%
%
%
%
%
%

\subsection{Further motivations and techniques}
In this section, we  give the history behind the ideas in this paper. Siegel in his study of the Hasse-Minkowski theorem generalized the classical holomorphic modular forms into Siegel modular forms. He  showed that  the   averaged representation number of a positive definite integral symmetric matrix $B_{n\times n}$ by the genus class of $A_{m \times m}$ is the $B$-th Fourier coefficient of the theta series associated to the genus class of $A$, which is a holomorphic Siegel modular form (Eisenstein series)~\cite{Siegel1}. Weil \cite{Weil,Weil2} gave a group theoretic interpretation of Siegel's work and introduced the oscillator representation of the metaplictic group  (double cover of the symplectic group). 

Shintani \cite{Shintani} used the oscillator representation and described  the Shimura correspondence \cite{Shimura} between the weight $k+1/2$ holomorphic modular forms  and the weight $2k$ holomorphic modular forms. Moreover, Shintani showed that the average of the integral weight modular forms $f$ over over the closed geodesics with discriminant $D$ (Weyl sums)  is the $D$-th Fourier coefficient of $\theta(f)$, where $\theta(f)$ is the theta transfer of $f$; see the work of  Katok and Sarnak~\cite{Katok} for the Maass forms.  In particular, the equidistibution of the CM points or closed geodesics of a given discriminant on the modular curve  follows from a sub-convex bound on the Fourier coefficients of the weight $1/2$ integral modular forms which was achieved by  Iwaniec~\cite{Henryk} for holomorphic  and Duke~\cite{Duke} for Maass forms. Our main observation is that by using the oscillator representation and the  spectral theory of the metaplictic group one can prove equidistribution results for the integral points  on the homogenous variety of an orthogonal group (a different group!).  

One aim of this paper is to generalize Shintani's  correspondence and give a correspondence from the classical automorphic forms of the orthogonal groups to the Siegel modular forms. We describe explicitly the weight (which is a finite dimensional representation of $GL_n(\mathbb{C})$) and the level of the associated Siegel modular form; see Proposition~\ref{ppww}. We also express  the Fourier coefficient of the associated Siegel modular forms in terms of the Weyl sums of the automorphic forms on the orthogonal group; see Theorem~\ref{weylfourier}. We use this identity to prove some new optimal  results for the distribution of the integral points on homogenous varieties.

  Ghosh, Gorodnik and Nevo \cite{GGN, GGN1,GGN2}  and Sarnak \cite{Sarnak2} used the spectral theory of automorphic forms  for proving some optimal results on the distribution of integral points on homogenous varieties if the associate automorphic spectrum satisfies the generalized Ramanujan conjecture \cite{conjectures}.  Our approach is different and give some optimal results which are not achievable by the previous methods. Our main idea is to generalize the work of Shintani to  the dual pairs of reductive groups $(G,G^{\prime})$ in a symplectic group~\cite{Howee} and relate the Weyl sums on a homogenous variety $X$ of $G$ to the period integrals of the image of the theta transfer of automorphic forms from $G$ to $G^{\prime}.$ The theta transfer has a large kernel, and as a result for all but a tiny fraction of automorphic forms of $G$, the associated Weyl sum is zero! For the remaning non-zero theta transfers, we use bounds on the generalized Ramanujan conjecture for $G^{\prime}$. This strategy gives some new optimal  results; see Theorem~\ref{spheredense} and Theorem~\ref{mthm},  only  if the automorphic spectrum of $G^{\prime}$ satisfies the Ramanujan conjecture  and not necessarily the automorphic spectrum of $G$! (or even the image of  the automorphic spectrum of $G$ under the theta transfer which lies inside the automorphic spectrum of  $G^{\prime}$ satisfies an average version of the generalized Ramanujan conjecture).  This is  a new feature of our work compare to the work of Ghosh, Gorodnik and Nevo \cite{GGN, GGN1,GGN2}  and Sarnak \cite{Sarnak2}; see Corollary~\ref{sggn} and the discussion after it.
%

 In this paper, we work with the dual pair $(G,G^{\prime})= (O_m, Sp_n) \subset Sp_{mn}(\mathbb{Q})$, where $O_m$ is compact at the archimedean place and  $X=M_{m\times n}$ is the  $m\times n$ matrices.  More concretely,    we use the oscillator representation in order to relate the distribution of the integral points on the representation variety of pairs of positive definite symmetric integral matrices, to  bound  the Fourier coefficient of the Siegel  modular forms. Bounding the Fourier coefficients of the classical modular forms has been extensively studied after Ramanujan's conjecture. The natural generalization of the weight $k$ holomorphic modular forms are the vector valued Siegle modular forms with a weight $\rho: GL_n(\mathbb{C})\to V_{\rho}$, where $\rho$ is a finite dimensional complex  representation. 
 Unfortunately, there are very few results on bounding the Fourier coefficients of the vector valued  Siegel modular forms with respect to a norm or a functional on $V_{\rho}$.   Kitaoka~\cite{Kitaoka,Kitaoka2} generalized the Kloosterman's method and  proved the analogue of the Kloosterman's  bound when $n=2$ and $\rho$ is one dimensional. B\"ocherer and  Raghavan \cite{Raghavan} generalized the Rankin-Selberg method for general $n$ and one dimensional $\rho.$ We refer the reader to the work of Kohnen \cite{Winfried} for further discussions and the expected optimal bound when $\rho$ is one dimensional.   It seems that the only known results are when $\rho$ is one dimensional. 
    This is partly caused by the lack of the interesting applications.  We give  some classical application of this problem. In particular, we show that an average version  of the Ramanujan bound on the Fourier coefficients of the vector valued Siegel modular forms implies the equdistribution of the integral points on
 the representation variety of pairs of quadratic forms at the optimal scale. In particular, our  results are optimal for $n=1$; see Theorem~\ref{spheredense} and Theorem~\ref{mthm}.


 \subsection*{Acknowledgements}\noindent
I would like to thank Prof. Simon Marshall, Prof. Zeev Rudnick, and  Prof. Peter Sarnak  for their comments on the earlier version of  this manuscript.

\section{Proof of Theorem~\ref{spheredense}}\label{sec2}
In this section we give a proof of Theorem~\ref{spheredense}. We use Theorem~\ref{Siegelthmm} and a proposition, which we formulate next.  Recall that $A_{m\times m}$ is a positive definite integral matrix and consider the lattice $(\mathbb{Z}^m,A).$  In the following propositon, we  give an upper bound on the number of root vectors  of $A$. Recall that $\vec{v} \in \mathbb{Z}^m$ is a root vector, if $\vec{v}^{\intercal}A\vec{v}=2 $ or 1.
\begin{prop}\label{rootprop}
The number of root vectors of $A_{m\times m}$ of length 1 and length $\sqrt{2}$ is less than $2m$ and $10m^2,$ respectively. 
\end{prop}
We give a proof of this proposition at the end of this section.   We assume this proposition and Theroem~\ref{Siegelthmm}, and proceed to give a proof of Theorem~\ref{spheredense}.

\subsection{Proof of Theorem~\ref{spheredense}}  
\begin{proof}We assume that $A$ is an even unimodular lattice. The proof for the odd unimodular lattice is similar, and we briefly discuss it at the end.   Let $R(N)$ be the representation mass of even integer $N$ by the genus class of $A$ that is defined in~\eqref{averep}. By the Siegel mass formula~\eqref{Siegelmf} and  the explicit formulas for the local densities; see \cite[Lemma 2]{Akshay}, we have 
\begin{equation}\label{siegelrep}
R(N)=\dfrac{mN^{m/2-1}\pi^{m/2}}{\Gamma(m/2+1)} (1+O(2^{-m/4})).
\end{equation}
By the Stirling's formula $\Gamma(m/2+1)= \sqrt{\pi m} \big(\frac{m}{2e} \big)^{m/2} (1+O(1/m)).$ By choosing $N=\lfloor \frac{m}{2\pi e}+\frac{\log(m)}{2\pi e}-1 \rfloor$, we have 
\[
R(N)\leq \dfrac{2\pi e}{e^{\pi e}\sqrt{\pi}} (1+O(1/m)) \leq 1.
\]
Let  $q$ be an odd number such that  $ 2q\leq \frac{m}{2\pi e}+\frac{\log(m)}{2\pi e} -t-1,$ where $0 \leq t=o(m).$ Then
\begin{equation}\label{ratioex}
\frac{R(2q+2)}{R(2q)} = \big(1+\frac{4\pi e}{m}\big)^{m/2-1} (1+o(1)) = e^{2\pi e} (1+o(1)) \geq 5113^2. 
\end{equation} 
Hence,
\[
R(2q) \leq  5113^{-t}.
\]
Note that 
\[
  \mu_{0}\big(\vec{x}^{\intercal}A\vec{x}=2q \text{ for some }  \vec{x} \in \mathbb{Z}^m \big)\leq \dfrac{\sum_{A_i\in C(A)} \frac{1}{|O_{A_i}(\mathbb{Z})|} R_{A_i}(2q)}{\sum_{A_i\in C(A)} \frac{1}{|O_{A_i}(\mathbb{Z})|}}=R(2q) \leq  5113^{-t}.
\]
This implies the first part of Theorem~\ref{spheredense}.  Next, we give a proof of the second part of Theorem~\ref{spheredense}.  Assume that $q\geq \frac{m}{2\pi e}+2.6\frac{\log(m)}{\pi e} +t$, where $0 \leq t=o(m).$  
%
%
We use the trivial point-pair invariant function $K(\vec{x},\vec{y})=1$ in \eqref{varind}, and obtain   
\[
\va(2q)=\dfrac{ \sum_{A_i} \frac{1}{|O_{A_i}(\mathbb{Z})|}\Big(\big(R_{A_i}(2q)-R(2q)\big)^2   \Big)}{ \sum_{A_i} \frac{1}{|O_{A_i}(\mathbb{Z})|}}.
\]
Suppose  that $R_{A_i}(q)=0$ for  $\alpha $ proportion of $A_i$ with respect to the Siegel mass probability.   Then, 
\begin{equation} \label{lowerbd}
\va(2q) \geq \alpha R(2q)^2.
\end{equation}
One the other hand, by Theorem~\ref{Siegelthmm}, and the fact that $K(\vec{x},\vec{y})=1$, which  implies  $\pi_{\infty}=1$ and  $\deg(\pi_{\infty})=0$ in the expansion of \eqref{mainform},  we have 
 \[ 
 \va(2q)=\sum_{\pi}   |\Theta(\phi_{\pi,j},2q)|^2,
 \]
 where the sum is over $\pi$ with $\pi_{\infty}=1.$
By Proposition~\ref{ppww} and ~\ref{Heckeosc},  $\Theta(\phi_{\pi,j})$ is a holomorphic cusp form of weight $m/2$ and level dividing 8. Hence, we have the following multiplicative relation
\[
\Theta(\phi_{\pi,j},2q)= \Theta(\phi_{\pi,j},2) \lambda_{\Theta(\phi_{\pi,j})}(q)
\]
where $\lambda_{\Theta(\phi_{\pi,j})}(q)$ is the $q$-th Hecke eigenvalue of $\Theta(\phi_{\pi,j}).$ Since $m$ is even,  by the Ramanujan bound on the Hecke  eigenvalues of homomorphic cusp forms, we have
\[
|\lambda_{\Theta(\phi_{\pi,j})}(q)|^2\leq d(q)^2q^{m/2-1},
\]
where $d(q)$ is the number of divisors of $q.$
Therefore, we have
\begin{equation}\label{var2q}
\begin{split}
\va(2q) = \sum_{\pi}   |\Theta(\phi_{\pi,j},2q)|^2&= \sum_{\pi} |\Theta(\phi_{\pi,j},2)|^2 |\lambda_{\Theta(\phi_{\pi,j})}(q)|^2
\\
&\leq d(q)^2 q^{m/2-1} \sum_{\pi} |\Theta(\phi_{\pi,j},2)|^2
\\
&=d(q)^2 q^{m/2-1}  \va(2),
\end{split}
\end{equation}
where we used 
\[
 \va(2)=\sum_{\pi}   |\Theta(\phi_{\pi,j},2)|^2.
\]
By definition, we have 
\begin{equation*}
\begin{split}
 \va(2)&\leq \dfrac{ \sum_{A_i} \frac{1}{|O_{A_i}(\mathbb{Z})|}\Big(\big(R_{A_i}(2)-R(2)\big)^2   \Big)}{ \sum_{A_i} \frac{1}{O_{A_i}}}
 \\
 &\leq  \max_{i}{R_{A_i}(2)}\dfrac{ \sum_{A_i} \frac{1}{|O_{A_i}(\mathbb{Z})|}\big|R_{A_i}(2)-R(2)\big| }{ \sum_{A_i} \frac{1}{O_{A_i}}}
 \\
 &\leq 2 \max_{i}{R_{A_i}(2)} R(2).
\end{split}
\end{equation*}
By Proposition~\ref{rootprop}, 
\[
\max_{i}{R_{A_i}(2)} \leq 10 m^2.
\]
Hence,
\[
 \va(2) \leq 20m^2 R(2)  .
\]
By equation~\eqref{var2q}, we have 
\[
\va(2q)\leq 20m^2  R(2)q^{m/2-1}d(q)^2.
\]
By the Siegel Mass formula in ~\eqref{siegelrep}, we have 
\[
R(2q)=R(2)q^{m/2-1}(1+O(2^{-m/4})).
\]
Hence,
\[
\va(2q)\leq 20m^2 R(2q)d(q)^2.
\]
We compare this upper bound with  the lower bound~\eqref{lowerbd} and obtain 
\[
\alpha R(2q)^2 \leq 20 R(2q)m^2d(q)^2.
\]
By the above and \eqref{ratioex}, we have
\[
\alpha\leq \frac{20m^2d(q)^2}{R(2q)}\ll \frac{m^{2+\epsilon}}{m^{2.1}(5113)^{t}}\leq (5113)^{-t},
\]
where we used the asymptotic formula $2q\sim \frac{m}{2\pi e}.$ This completes the proof of our theorem for even unimodular lattices. The argument for odd unimodular lattices
is similar. The improved bound in the case of the odd unimodular is due to our upper bound $2m$ on the number of root vectors of length 1 in Proposition~\ref{rootprop}. 
\end{proof}
%
%
%
%
%

\subsection{Proof of Proposition~\ref{rootprop}}
We begin by proving some auxiliary lemmas.

\begin{lem}\label{count} Assume that the root vectors of norm $1$ of $A$ spans $\mathbb{Z}^m$. Then $A$ is isomorphic to the identity matrix $I$. Moreover, then number of root vectors of length $1$ is $2m,$ and the number of the root vectors of norm $\sqrt{2}$ is $2m(m-1)$.
\end{lem}
\begin{proof}
Let $B:=\{ \vec{v}_1,\dots, \vec{v}_m\}$ be a basis of root vectors such that $\vec{v}_i^{\intercal}A\vec{v}_i=1.$ By the Cauchy-Schwarz inequality, we have 
\[
|\vec{v}_i^{\intercal} A\vec{v}_j|< \big(|\vec{v}_i^{\intercal} A\vec{v}_i||\vec{v}_j^{\intercal} A\vec{v}_j|   \big)^{1/2}=1,
\]
for $i\neq j.$ Since  $\vec{v}_i$ and $A$ are integral for $1\leq i\leq m$, it follows that $\vec{v}_i^{\intercal} A\vec{v}_j=0.$ Hence, $B$ is an orthonormal basis, which implies  $A$ is isomorphic  to $I.$ It is easy to check that the root vectors of length 1 are  $\{ \pm \vec{v}_1,\dots, \pm \vec{v}_m\}.$ By a simple counting,  the number of root vectors of length $\sqrt{2}$ is $2m(m-1).$ This completes the proof of our lemma.   
\end{proof}
\begin{lem}\label{root2}
Assume that $A$ does not have any root vector of length 1, and there exists an orthogonal basis of root vectors of length $\sqrt{2}$ (for $\mathbb{R}^m$ not necessarily the lattice!)
Then the number of the root vectors of length $\sqrt{2}$ is less than $6m^2-4m$.
\end{lem}
\begin{proof}
Let $B:=\{ \vec{v}_1,\dots, \vec{v}_m\}$ be an orthogonal  basis of root vectors of length $\sqrt{2}.$ Let $\vec{u}\notin \pm B$ be a root vector of length $\sqrt{2}.$ Let $[\vec{u}]:=(\vec{u}^{\intercal}A\vec{v}_j)\in \mathbb{Z}^m$, where $\vec{v}_j\in B.$ Let $\#\vec{u}$ denote the number of non-zero coordinates of $[\vec{u}].$  By the plancherel identity, we have 
\[
2= \sum_{\vec{v}_j\in B} \frac{|\vec{u}^{\intercal}A\vec{v}_j|^2}{2}.
\]
Since  $\vec{u}\notin \pm B$, it follows that $\#\vec{u}=4$ and $\vec{u}^{\intercal}A\vec{v}_j=\pm 1$ or 0. Let $S:=\{\vec{u}_1, \dots,\vec{u}_{R}\}$ denote the set of all root vector $\vec{u}$ of length $\sqrt{2}$, where $\vec{u}\notin \pm B.$ Let $M_{R\times m}:=[\vec{u}_i^{\intercal}A\vec{v}_j]$, for $\vec{u}_i\in S$ and $\vec{v}_j\in B.$ In what follows, we given an upper bound on $R.$ Each row $[\vec{u_i}]$ contains exactly four   $\pm 1$ and zero at other entries.  So, the matrix $M$ contains $4R$ nonzero elements. By a pigeon-hole argument there exits a column, associated to $\vec{v}_j$ for some $j$, which contains at least $\frac{4R}{m}$ non-zero elements which are $\pm1$. Without loss of generality, suppose that 
$\vec{u}_i^{\intercal}A\vec{v}_1=\pm 1$ for $1\leq i\leq \frac{4R}{m}.$ By the plancherel identity,
\[
\vec{u}_i^{\intercal}A\vec{u}_j=1/2+ \sum_{k>1} \frac{(\vec{u}_i^{\intercal}A\vec{v}_k)(\vec{u}_j^{\intercal}A\vec{v}_k)}{2}\in \mathbb{Z},
\]
where $1\leq i,j\leq \frac{4R}{m}.$ The integrality of the inner product implies $\vec{u}_i$ and $\vec{u}_j$ are non-zero at either 1 or 3 other columns. Without loss of generality assume that $\vec{u}_1^{\intercal}A\vec{v}_j\neq 0$ for $1 \leq j\leq 4.$ Then for $1\leq i\leq  \frac{4R}{m}$  $\vec{u}_i^{\intercal}A\vec{v}_j\neq 0$ for some $2\leq  j\leq 4.$ By a pigeon-hole argument for some $2\leq j\leq 4$ there are more than $\frac{4R}{3m}$, $\vec{u}_i$ such that $\vec{u}_i^{\intercal}A\vec{v}_j\neq 0.$ Without loss of generality assume that  $1\leq i\leq \frac{4R}{3m},$ we have $\vec{u}_i^{\intercal}A\vec{v}_2\neq 0.$  Finally by the integrality of $\vec{u}_i^{\intercal}A\vec{u}_j$, it follows that
\[
\frac{4R}{3m}\geq 8(m-2).
\]
This implies \(R\leq 6 m(m-2)   \). So the total number of root vectors of length $\sqrt{2}$ is less than $6m^2-4m.$ This completes the proof of our lemma.

%
\end{proof}

\begin{lem}\label{r2}
Assume that $A$ does not have any root vector of length 1. Then the number of the root vectors of length $\sqrt{2}$ is less than $10m^2-4m$.
\end{lem}
\begin{proof}
Let $T=\{\vec{w}_1,\dots,\vec{w}_q\}$ be a maximal set of orthogonal root vectors of norm $\sqrt{2}.$ By Lemma~\ref{root2}, the number of the root vectors which are in the span of $T$ is less than $6q^2-4q$. We proceed and show that the number of root vectors  which are not in the span of $T$ is less than $4q^2.$

 Suppose that  $\vec{u}$ is a root vector, and $\vec{u}\notin \text{span}(T).$ Let $[\vec{u}]:=(\vec{u}^{\intercal}A\vec{w}_i)$, where $\vec{w}_i\in T.$ Let $\#\vec{u}$ denote the number of non-zero coordinates of $[\vec{u}].$  We show that $\#\vec{u}= 1 , 2.$ By the plancherel inequality, we have 
 \[
 2=\vec{u}^{\intercal}A\vec{u} \geq \sum_{\vec{w}_i\in T} \frac{|\vec{u}^{\intercal}A\vec{w}_i|^2}{2}.
 \]
 This shows that $\#\vec{u}\leq 4$. 
 The maximality assumption of $T$ excludes  $\#\vec{u}=0$, and $\vec{u}\notin \text{span}(T)$ excludes  $\#\vec{u}=4$.  Suppose that $\#\vec{u}=3$ and   $\vec{w}_1$, $\vec{w}_2$ and $\vec{w}_3$ have  non-zero  inner product with $\vec{u}$. Then we define  
 \[
 \vec{u}^{\prime}:=2\vec{u}-((\vec{u}^{\intercal}A\vec{w}_1) \vec{w}_1+(\vec{u}^{\intercal}A\vec{w}_2) \vec{w}_2+(\vec{u}^{\intercal}A\vec{w}_3) \vec{w}_3)\in \mathbb{Z}^m. 
 \]
We have \( {\vec{u}^{\prime}}^{\intercal}A \vec{u}^{\prime}=2 \). Note  that $ \vec{u}^{\prime}$ is a root vector and is orthogonal to all vectors in $T,$ which contradicts with the maximality of $T.$ So, the only possibilities for $\#\vec{u}$ are 1 or 2. 

 Let $m(\vec{f})$ be the number of roots vectors $\vec{v}\notin \text{span}(T)$ such that $[\vec{v}]=\vec{f}.$  We claim that $m(\vec{f})\leq 2$. First, suppose  that $\#\vec{f}=2,$ and without loss of generality $\vec{f}=(1,1,0,\dots,0).$ Assume the contrary that $m(\vec{f})>2.$ Then for some root vectors $\vec{u}_1$, $\vec{u}_2$ and $\vec{u}_3$  outside $\text{span}(T)$, we have   $[\vec{u_1}]=[\vec{u_2}]=[\vec{u_3}]=\vec{f}$.  We show that $\vec{u}_i$ are orthogonal to each other.  Assume the contrary that  $\vec{u}_1^{\intercal} A\vec{u}_2\neq 0.$ Then,    $\vec{u}_1^{\intercal} A\vec{u}_2= \pm 1.$ Assume that  $\vec{u}_1^{\intercal} A\vec{u}_2= - 1.$ Then $\vec{u}_1 +\vec{u}_2$ is a root vector. By plancherel inequality, we have 
\[
2=(\vec{u}_1 +\vec{u}_2)^{\intercal} A(\vec{u}_1 +\vec{u}_2)\geq \sum_{\vec{w}_i\in T} \frac{|(\vec{u}_1 +\vec{u}_2)^{\intercal}A\vec{w}_i|^2}{2}=2\#\vec{u} \geq 4,
\]
which is a contradiction. So $\vec{u}_1^{\intercal} A\vec{u}_2= 1$, and $\vec{u}_1-\vec{u}_2$ is a root vector which is orthogonal to $T.$ This contradicts with the maximality of $T$. Hence $\vec{u}_1$, $\vec{u}_2$ and $\vec{u}_3$ are orthogonal root vectors.  Substitute  $\vec{w}_1$ and $\vec{w}_2$ from $T$ with $\vec{u}_1$, $\vec{u}_2$ and $\vec{u}_3.$ Then the new set is an orthogonal set of root vectors which has 1 more element than $T.$
 This contradicts with the maximality of $T.$ This shows that $m(\vec{f})\leq 2$ when $\#\vec{f}=2.$ 

Next, suppose that $\#\vec{f}=1$ and without loss of generality  $\vec{f}=(1, 0, \dots,0).$ Similarly, assume the contrary that $\vec{u}_1$, $\vec{u}_2$ and $\vec{u}_3$ are root vectors outside $\text{span}(T)$ with  $[\vec{u_1}]=[\vec{u_2}]=[\vec{u_3}]=\vec{f}.$  We claim that  $\vec{u}_i^{\intercal}A\vec{u}_j\neq 0$ for every $1\leq i, j\leq 3.$ Assume the contrary that $\vec{u}_1$ and $\vec{u}_2$ are orthogonal to each other then substitute  $\vec{u}_1$ and $\vec{u}_2$  with $\vec{w}_1$. This again  contradicts with the maximality of $T.$ Next, we show that  $\vec{u}_1^{\intercal}A\vec{u}_2\neq 1.$ Otherwise, $\vec{u}_1-\vec{u}_2$ is a root vector which is orthogonal to all vectors in $T$ and this also contradicts with the maximality of $T.$ Hence $\vec{u}_i^{\intercal}A\vec{u}_j= -1$ for every $i\neq j.$ Then $\vec{u}_1 +\vec{u}_2$ is a root vector. By plancherel inequality, we have 
\[
2=(\vec{u}_1 +\vec{u}_2)^{\intercal} A(\vec{u}_1 +\vec{u}_2)\geq \frac{|(\vec{u}_1 +\vec{u}_2)^{\intercal}A\vec{w}_1|^2}{2}=2,
\]
which implies $\vec{u}_1 +\vec{u}_2=\vec{w}_1.$ Similarly, we have $\vec{u}_1 +\vec{u}_3=\vec{w}_1,$ and $\vec{u}_2 +\vec{u}_3=\vec{w}_1.$ This implies
$\vec{u}_1=\vec{u}_1=\vec{u}_1= \vec{w}_1/2$ which is a contradiction. Therefore $m(\vec{f})\leq 2$ for every $f\in \mathbb{Z}^q.$

Note that there are at most $2q$ vectors $\vec{f}\in \{0,\pm 1\}^q$ with $\#\vec{f}=1$. Since $m(\vec{f})\leq 2$, there are at most $4q$ root vectors $\vec{u}\notin \text{span } T$ such that $\#[\vec{u}]=1.$ Similarly, there are  at most $2q(q-1)$ vectors $\vec{f}\in\{0,\pm 1\}^q$ with $\#\vec{f}=2,$ and that implies  there are at most $4q(q-1)$ root vectors $\vec{u}\notin \text{span } T$ such that $\#[\vec{u}]=2.$ Therefore, the total number of root vectors $\vec{u}$, where $\vec{u}\notin \text{span}(T),$ is less than $4q^2.$ This completes the proof of our lemma.
\end{proof}
Finally, we give a proof of Proposition~\ref{rootprop}.
\begin{proof} [Proof of Proposition~\ref{rootprop}]
Let $S:=\{\pm \vec{v}_1,\dots,\pm \vec{v}_p\}$ be the set  integral vectors such that $\vec{x}^{\intercal}A\vec{x}=1.$ By Lemma~\ref{count}, it follows that $\{ \vec{v}_1,\dots, \vec{v}_p\}$ is an orthonormal set of vectors. Hence, the number of root vectors of length 1 is less than $2m.$ 
 Let $V:=\text{span}_{\mathbb{Z}} \{ \vec{v}_1,\dots, \vec{v}_p\}\subset \mathbb{Z}^m$ be the lattice generated with the root vectors of length 1. It is easy to see that $\mathbb{Z}^n=V\oplus V^{\perp},$ where $V^{\perp}\subset \mathbb{Z}^m$ is the   orthogonal complement of $V \subset \mathbb{Z}^m$ with respect to $A.$ By our assumption, all the root vectors of $V^{\perp}$ has length $\sqrt{2}$ (there is no root vector of length 1 in $V^{\perp}$). Moreover, if $u$ is any root vector with length $\sqrt{2}$ then either $u\in V$ or $u\in V^{\perp}.$ By Lemma~\ref{count},  the number of root vectors of length $\sqrt{2}$ in $V$ is less than $2p(p-1).$ By Lemma~\ref{r2}, the number of root vectors of length $\sqrt{2}$ is less than $10(m-p)^2.$ This completes the proof of our Proposition.

\end{proof}

\section{Proof of Theorem~\ref{mthm}}\label{sec3}

 Recall the notations while  formulating  Theorem~\ref{mthm} and Theorem~\ref{Siegelthmm}.  In this section, we assume that $F(x_1,\dots,x_m)=\vec{x}^{\intercal} A\vec{x}$ for some positive definite  symmetric matrix $A$,   where $\vec{x}=\begin{bmatrix}x_1\\ \vdots \\ x_m \end{bmatrix}.$ 
 We give a sharp upper bound on $ \va(N,\eta)$ by assuming Theorem~\ref{Siegelthmm}.  
\subsection{Scaling the  point-pair invariant function}
We prove a simple lemma which relates the point-pair invariant functions $K_{\eta}(\vec{x},\vec{y})$ (defined in   Theorem~\ref{mthm}) to   $K_{r,N}(\vec{x},\vec{y})$ (defined in \eqref{pp}). Recall that
\(K_{\eta}(\vec{x},\vec{y}):=C_{\eta}k\big(\frac{\sqrt{F(\vec{x}-\vec{y})}}{\eta}\big),\) where $\vec{x},\vec{y}\in V_1(\mathbb{R}),$ $\eta\in \mathbb{R}$ and $C_{\eta}$ is a normalization factor such that
\(
\int_{ V_1(\mathbb{R})} K_{\eta}(\vec{x},\vec{y}) d \mu(\vec{y})=1.
\)
Moreover,  for $\vec{x}, \vec{y}\in V_{N,A_i}(\mathbb{R})$, where  $A_i\in C(A)$, we defined 
\(
K_{r,N}(\vec{x},\vec{y}):= C_{N,r}k\big(\frac{|\vec{x}-\vec{y}|_i}{r}\big),
\)
where $|\vec{x}-\vec{y}|_i:=\sqrt{(\vec{x}-\vec{y})^{\intercal}A_i(\vec{x}-\vec{y})}$, and 
\[
\int_{V_{A_i,N}(\mathbb{R})}K_{r,N}(\vec{x},\vec{y}) d\mu_{i,N}(\vec{y})=1,
\]
where $\mu_{i,N}$ is the Haar  probability measure on $V_{A_i,N}.$
\begin{lem} Assume that $\vec{x},\vec{y} \in V_{A,N}(\mathbb{R}).$ 
We have 
\[
K_{r,N}(\vec{x},\vec{y})=K_{\eta}( \frac{\vec{x}}{\sqrt{N}},\frac{\vec{y}}{\sqrt{N}}),
\]
where $\eta=\frac{r}{\sqrt{N}}.$
\end{lem}
\begin{proof} Note that the probability measure $\mu_{N}$ on $V_{A, N}(\mathbb{R})$ is the pull back of the probability measure  $\mu$ on $V_{A,1}(\mathbb{R})$ by the scaling map with $1/\sqrt{N}$. Hence,
we have 
\[
K_{r,N}( \vec{x}, \vec{y})=C_{N,r}k\big(\frac{|\vec{x}-\vec{y}|}{r}\big),
\]
where $|\vec{x}-\vec{y}|=\sqrt{(\vec{x}-\vec{y})^{\intercal}A(\vec{x}-\vec{y})}=\sqrt{F(\vec{x}-\vec{y})}.$ The lemma follows from the above identity and the definition of $K_{\eta}(\vec{x},\vec{y})$.  
\end{proof}

%
%
%
\subsection{Proof of Theorem~\ref{mthm}}
\begin{proof} Let $\{\phi_{k,i}   \}$ be an orthonormal  basis of automorphic forms which are harmonic polynomials of degree $k$ in  $L^2\big(O_A(\mathbb{Q})\backslash O_A(\mathbb{A}_{\mathbb{Q}})/O(\prod_p\mathbb{Z}_p)\big).$  We consider them as harmonic polynomials on the  disjoint union of the following quadrics 
\[\bigcup\big\{(\vec{X},A_i): \vec{X}\in O_{A_i}(\mathbb{Z})\backslash V_{A_i,1}(\mathbb{R}),  A_i \in C(A) \big\}.
\]
By Theorem~\ref{Siegelthmm},
we have
\[
\va(N,r)=\sum_{k=1}^{\infty} h_{r}(k)^2 N^{-k}\Theta(\phi_{k,i},N)^2,
\]
where $\Theta(\phi_{k,i},N)$ is the $N$-Fourier coefficient of $\Theta(\phi_{k,i})$ that is  the theta transfer of $\phi_{k,i}$. By Proposition~\ref{ppww} and ~\ref{Heckeosc},  $\Theta(\phi_{k,i})$ is a Hecke holomorphic modular form of weight $m/2+k$ and level dividing $4|A|$.  Recall that $\gcd(N,|A|)=1$. Hence,  by the multiplicative property of the Fourier coefficients, we have 
\[
\Theta(\phi_{k,i},N)=\lambda_{\Theta(\phi_{k,i})}(N)\Theta(\phi_{k,i},1),
\]
where $\lambda_{\Theta(\phi_{k,i})}(N)$ is the $N$-th Hecke eigenvalue of $\Theta(\phi_{k,i}).$
By the assumptions of Theorem~\ref{mthm} the Ramanujan bound holds for $\lambda_{\Theta(\phi_{k,i})}(N)$, and  we have 
\[
\lambda_{\Theta(\phi_{k,i})}(N)\ll N^{\frac{k+m/2-1}{2}+\epsilon},
\]
where the implied constant involved in $\ll$ only depends on $\epsilon$. Hence, 
\begin{equation}\label{multvar}
\begin{split}
\va(N,r)&=\sum_{k=1}^{\infty} h_{N,r}(k)^2 N^{-k}\Theta(\phi_{k,i},N)^2
\\
&\ll N^{m/2-1+\epsilon} \sum_{k=1}^{\infty} h_{N,r}(k)^2 \Theta(\phi_{k,i},1)^2
\end{split}
\end{equation}
It follows from the definition  of the spherical transform $h_{N,r}(k)$ in  \eqref{sphtrans} that 
\[
h_{N,r}(k)= h_{1,\frac{r}{\sqrt{N}}}(k).
\]
Hence, by Theorem~\ref{Siegelthmm}, we have 
\[
 \va(1,\frac{r}{\sqrt{N}})=\sum_{k=1}^{\infty} h_{N,r}(k)^2 \Theta(\phi_{k,i},1)^2.
\]
By substituting the above in \eqref{multvar}, we obtain
\[
\va(N,r) \ll N^{m/2-1+\epsilon}  \va(1,\frac{r}{\sqrt{N}}).
\]
Next, we give an upper bound on $\va(1,\frac{r}{\sqrt{N}}).$ For simplicity we write $\eta=\frac{r}{\sqrt{N}}$ and $K_{\eta}(\vec{x},\vec{y})=K_{\frac{r}{\sqrt{N}},1}(\vec{x},\vec{y})$ for $\vec{x},\vec{y}\in V_{A_i,1}(\mathbb{R}) $.  By \eqref{Siegelvarfor}, we have 
\begin{equation*}
\begin{split}
\va(1,\eta)&\leq \max_{i} \Big(\int_{V_{A_i,1}(\mathbb{R})} \Big( \big(\sum_{\vec{Y}\in  V_{A_i,1}(\mathbb{Z}) }K_{\eta}(\vec{x},\vec{y})\big)- R(1)  \Big)^2  d\mu_i(\vec{x}) \Big)
\\
&\leq \sup_{i,\vec{x}} \Big( \big(\sum_{\vec{Y}\in  V_{A_i,1}(\mathbb{Z}) }K_{\eta}(\vec{x},\vec{y})\big)- R(1)  \Big) \int_{V_{A_i,1}(\mathbb{R})} \Big| \big(\sum_{\vec{Y}\in  V_{A_i,1}(\mathbb{Z}) }K_{\eta}(\vec{x},\vec{y})\big)- R(1)  \Big|  d\mu_i(\vec{x}),
\\
&\leq  C_{1,\eta} (\max_{i}R_{A_i}(1)+R(1)) \ll C_{1,\eta} \ll \eta^{-(m-1)}.
\end{split}
\end{equation*}
Therefore, we have
$$
\va(N,r) \ll \frac{N^{m/2-1+\epsilon}}{\eta^{m-1}}.
$$
By assuming $m\geq 4,$ and bounding the local densities in Hardy-Littlewood formula \cite[Remark 1.7]{Optimal}, we have
\[
N^{m/2-1+\epsilon} \ll N^{\epsilon} R_F(N).
\]
Therefore,
\[
\va(N,r)\ll \frac{N^{\epsilon} R_F(N)}{\eta^{m-1}},
\]
where  the implicit constant in $\ll$ only depends only on $A$ and $\epsilon.$ This completes the proof of Theorem~\ref{mthm}.
\end{proof}
\subsection{Proof of Corollary~\ref{sthm}}
\begin{proof} Assume that $\eta \ll R_F(N)^{-\frac{1}{(m-1)}-\epsilon}.$  We prove the second part of the corollary. Let 
\(
E(N,\eta):= \{  \vec{x}\in V_1(\mathbb{R}): \forall \vec{y}\in  \frac{1}{\sqrt{N}}V_N(\mathbb{Z}), \sqrt{F(\vec{x}-\vec{y})}\geq \eta\}.
\)
Since $k$ is positive and  $k(x)=1$ for $x\in(-1,1),$ we have 
\[
 \big(1-\mu(E(N,\eta))\big) C_{\eta} \leq  \int_{V_1(\mathbb{R}) } \sum_{\vec{y}\in  \frac{1}{\sqrt{N}} V_N(\mathbb{Z})} K_{\eta}(\vec{x},\vec{y}) d\mu(\vec{x})=R_F(N). 
\]
It is easy to check that $C_{\eta}\gg \eta^{-(m-1)}.$ Therefore, by \eqref{lohardy}
\[
 \big(1-\mu(E(N,\eta))\big) \ll R_F(N) \eta^{(m-1)} \ll R_F(N)^{-(m-1)\epsilon}  \ll N^{-\epsilon}.
\]
This completes the proof of the second part of the corollary. Next, assume that  $\eta \gg R_F(N)^{-\frac{1}{(m-1)}+\epsilon}$. 
By Chebyshev's inequality and Theorem~\ref{mthm}
\[
\mu(E(N,\eta)) R_F(N)^2\leq \va_F(N,\eta) \ll \frac{N^{\epsilon} R_F(N)} {\eta^{m-1}}.
\]
Therefore, by \eqref{lohardy}
\begin{equation}\label{chebche}
\mu(E(N,\eta))  \ll \frac{N^{\epsilon}} {\eta^{m-1} R_F(N)}  \ll N^{-\epsilon} .
\end{equation}
This completes the proof of the corollary. 
\end{proof}
\subsection{Proof of Corollary~\ref{sggn}}
\begin{proof}
Recall that 
\(
 \bar{K}_{m}:=\lim_{\delta\to 0}\limsup_{\epsilon_0 \to 0}\frac{\log \big(\#S^{m-1}_{N_{\delta,\epsilon_0}}(\mathbb{Z}) \big) }{\log\big( 1/\text{vol }(C(\vec{x},\epsilon_{0}))\big)},
\) and $\bar{K}_{m}\geq 1.$ It is enough to show that $\bar{K}_{m} \leq 1+\epsilon$ for any $\epsilon >0.$
 We have $\text{vol }(C(\vec{x},\epsilon_{0}) \gg \epsilon_0^{m-1}.$ Since $m$ is even then Theorem~\ref{mthm} and  inequality~\eqref{chebche} holds unconditionally  for $S^{m-1}$, and we have 
 \[
\mu(E(N,\epsilon_0))  \ll \frac{N^{\epsilon}} {\text{vol }(C(\vec{x},\epsilon_{0}) \#S^{m-1}_{N}(\mathbb{Z})}.
 \]
 Hence, by the definition of $N_{\delta,\epsilon_0},$ we have 
 \[
\#S^{m-1}_{N_{\delta,\epsilon_0}}(\mathbb{Z}) \leq \text{vol }(C(\vec{x},\epsilon_{0})^{(-1-\delta/(m-1)-\epsilon)}.
 \]
 Therefore, 
 \[
 \bar{K}_{m}:=\lim_{\delta\to 0}\limsup_{\epsilon_0 \to 0}\frac{\log \big(\#S^{m-1}_{N_{\delta,\epsilon_0}}(\mathbb{Z}) \big) }{\log\big( 1/\text{vol }(C(\vec{x},\epsilon_{0}))\big)}\leq  \lim_{\delta\to 0}{(1+\delta/(m-1)-\epsilon)} \leq 1+\epsilon.
 \]
 for every $\epsilon>0$. This concludes the proof of Corollary~\ref{sggn}.

\end{proof}

\subsection{Proof of Corollary~\ref{strongGGN}}
\begin{proof}
The proof is based on a Borel-Cantelli argument. Define
\[
A_{k,\epsilon}:=\left\{x\in S^{m-1}(\mathbb{R}): |x-z|>\epsilon , \forall z\in S^{m-1}(\mathbb{Z}[1/p]) \text{ with } H(z)\leq p^{k}  \right\}.
\]
We note that there is a one-to-one correspondence between  $z\in S^{m-1}(\mathbb{Z}[1/p]) \text{ with } H(z)\leq  p^{k}$ and the integral points  $\frac{1}{p^k}S^{m-1}_{p^{2k}}(\mathbb{Z}).$
Note that by using  $N=p^{2k}$ in Theorem~\ref{mthm}, it follows that inequality~\eqref{chebche} holds unconditionally  for $\frac{1}{p^k}S^{m-1}_{p^{2k}}(\mathbb{Z})$, and we have 
\[
\mu(A_{k,\epsilon})\ll \frac{p^{2k\epsilon}} {\epsilon^{m-1} p^{k(m-2)} }.
\]
Let $B_{k,\delta}:=A_{k,p^{-k(\frac{m-2}{m-1}-\delta)}}.$ By the above inequality $\mu(B_{k,\delta})\ll p^{-k(\delta(m-1)-2\epsilon)} $ for any $\epsilon>0.$ Note that $\sum_{k} \mu(B_{k,\delta})\ll \sum_{k}p^{-k(\delta(m-1)-2\epsilon)} \ll \infty. $ By the Borel-Cantelli lemma, for almost all $x\in S^{m-1}(\mathbb{R}),$ there exists $k_{x}$ such that $x\notin B_{k}$ for every $k>k_{x}.$ In other words, for almost every $x\in S^{m-1}(\mathbb{R}),$ $\delta>0$, and $\varepsilon \in (0,\varepsilon_0(x,\delta)),$ there exists $z \in  S^{m-1}(\mathbb{Z}[1/p]) $ such that 
\(
|x-z|_{\infty}\leq \epsilon \text{ and } H(z)\leq \epsilon^{-\frac{m-1}{m-2}-\delta}.\)  This concludes the proof of Corollary~\ref{strongGGN}.\end{proof}

\section{The Siegel variance formula}
Recall the definition of the Siegel variance sum $\va(B,r)$ in \eqref{Siegelvarfor}.  In this section, we give an adelic integration formula for  $\va(B,r).$
First, we write $\va(A_i,B,r)$ and $\va(B,r)$ in terms of the $O_{A_i}(\mathbb{Z})$ orbits of $V_{A_i,B}(\mathbb{Z}).$
We define the $O_{A_i}(\mathbb{Z})$ invariant function 
\[\tilde{K}_{r}(\vec{X},\vec{Y}):= \sum_{\gamma\in O_{A_i}(\mathbb{Z})}  K_{r,B}(\gamma \vec{X},\vec{Y}), \]
where $ \vec{X},\vec{Y} \in O_{A_i}(\mathbb{Z})\backslash  V_{A_i,B}(\mathbb{Z}). $

\begin{lem}\label{easylem}
We have 
\[
\frac{\va(A_i,B,r)}{|O_{A_i}(\mathbb{Z})|}= \int_{O_{A_i}(\mathbb{Z})\backslash V_{A_i,B}(\mathbb{R})} \Big(\big( \sum_{\vec{Y}\in O_{A_i}(\mathbb{Z})\backslash  V_{A_i,B}(\mathbb{Z}) } \frac{\tilde{K}_{r}(\vec{X},\vec{Y})}{|O_{A_i,\vec{Y}(\mathbb{Z})}|}\big)- R_{A_i}(B)  \Big)^2  d
\mu_i(\vec{X}),
\]
where  $|O_{L_i,\vec{Y}(\mathbb{Z})}|$ is the size of the stabilizer of $\vec{Y}$ in $O_{A_i}(\mathbb{Z}).$
\end{lem}
\begin{proof}
We have
\[
\sum_{\vec{Y}\in  V_{A_i,B}(\mathbb{Z}) }K_{r,B}(\vec{X},\vec{Y})= \sum_{\vec{Y}\in O_{A_i}(\mathbb{Z})\backslash  V_{A_i,B}(\mathbb{Z}) } \frac{\tilde{K}_{r}(\vec{X},\vec{Y})}{|O_{A_i,\vec{Y}(\mathbb{Z})}|}.
\]
Therefore, 
\begin{equation*}
\begin{split}
\va(A_i,B,r)=
 \int_{V_{A_i,B}(\mathbb{R})} \Big( \big(\sum_{\vec{Y}\in O_{A_i}(\mathbb{Z})\backslash  V_{A_i,B}(\mathbb{Z}) } \frac{\tilde{K}_{r}(\vec{X},\vec{Y})}{|O_{A_i,\vec{Y}(\mathbb{Z})}|}\big)- R_{A_i}(B)  \Big)^2  d\mu_i(\vec{X}).
\end{split}
\end{equation*}
The lemma follows from the fact that $\tilde{K}_{r}(\vec{X},\vec{Y})$ is $O_{A_i}(\mathbb{Z})$ invariant on the $\vec{X}$ variable  and $O_{A_i}(\mathbb{Z})\backslash V_{A_i,B}(\mathbb{R})$ is a fundamental domain for this action. 
\end{proof}

\begin{lem}\label{lem22}
We have
\begin{equation*}
\va(B,r)=\dfrac{\sum_{A_i} \int_{O_{A_i}(\mathbb{Z})\backslash V_{A_i,B}(\mathbb{R})} \Big(\big( \sum_{Y\in O_{A_i}(\mathbb{Z})\backslash  V_{A_i,B}(\mathbb{Z}) } \frac{\tilde{K}_{r}(\vec{X},\vec{Y})}{|O_{L_i,\vec{Y}(\mathbb{Z})}|}\big)- R(B)  \Big)^2  d
\mu_i(\vec{X}) }{{\sum_{A_i} \frac{1}{|O_{A_i}(\mathbb{Z})|}}}.
\end{equation*}
\end{lem}
 \begin{proof}  By \eqref{normalization}, we have 
 \[
 \int_{V_{A_i,B}(\mathbb{R})} \sum_{\vec{Y}\in  V_{A_i,B}(\mathbb{Z}) }K_{r,B}(\vec{X},\vec{Y}) d\mu_i(\vec{X})=R_{A_i}(B).
 \]
 Hence, by lemma~\ref{easylem}
 \begin{multline*} \dfrac{ \va(A_i,B,r)+\big(R_{A_i}(B)-R(B)\big)^2}{O_{A_i} (\mathbb{Z})}\\
  =  \int_{O_{A_i}(\mathbb{Z})\backslash V_{A_i,B}(\mathbb{R})} \Big(\big( \sum_{Y\in O_{A_i}(\mathbb{Z})\backslash  V_{A_i,B}(\mathbb{Z}) } \frac{\tilde{K}_{r}(\vec{X},\vec{Y})}{|O_{L_i,\vec{Y}(\mathbb{Z})}|}\big)- R(B)  \Big)^2  d
\mu_i(\vec{X}).
\end{multline*}
By summing both side of the above identity over $A_i$ and dividing by ${\sum_{A_i} \frac{1}{|O_{A_i}(\mathbb{Z})|}}$, we conclude the lemma. 
%
%
\end{proof}
\subsection{The Siegel variance as an adelic integral}\label{adelicsec}
In this section, we give a formula for $\va(B,r)$ in terms of an integral over a double quotient of the adelic orthogonal group.
\subsubsection{Adelic point-pair invariant function}
We extend the point-pair invariant function $K_{r,B}(\vec{X},\vec{Y})$ (defined in \eqref{pp}) into an automorphic point-pair invariant function on the adelic points of the orthogonal group $O_A$. We begin by defining the adelic points of the orthogonal group $O_A.$
%
Let $O_A(R)$ denote the orthogonal group of $A$ with coefficients in a commutative ring $R,$ which we consider as a subset of $ GL_m(R):$
\[
O_A(R):=\big\{X\in GL_m(R):  X^{\intercal} AX =A  \big\}.
\]
Let $\mathbb{A}_f=\hat{\prod}_p^{\mathbb{Z}_p} \mathbb{Q}_p$ be the ring of finite adeles which is the restrictive direct  product of $\mathbb{Q}_p$ with respect to $\mathbb{Z}_p.$ Let 
$\mathcal{L}_{A,B}$ denote the space of $(\vec{X},L),$ where  $\vec{X}\in  V_{A,B}(\mathbb{R})$ and  $L\subset \mathbb{Q}^m$ is a lattice  where $(L,A)$ has the same genus as $(\mathbb{Z}^m,A):$
\[
\mathcal{L}_{A,B}:=\big\{(\vec{X},L): \vec{X}\in V_{A,B}(\mathbb{R}), L\subset \mathbb{Q}^m, \text{ and }(L\otimes \mathbb{Z}_p,A) \sim (\mathbb{Z}_p^d,A), \forall \text{ prime } p\big\},
\]
where $(L\otimes \mathbb{Z}_p,A) \sim (\mathbb{Z}_p^d,A)$ means there exists $g\in O_A(\mathbb{Q}_p) $ such that $g\mathbb{Z}_p^d=L\otimes \mathbb{Z}_p.$
Note that $O_A(\mathbb{A}_{\mathbb{Q}})$ acts transitively on $\mathcal{L}_{A,B}$ by: 
\[
(g_{\infty},\prod_p g_p). (\vec{X},L):= \Big(g_{\infty} \vec{X}, \big( \prod_p g_p (L\otimes \mathbb{Z}_p)\big) \cap \mathbb{Q}^n\Big),
\]
where $(g_{\infty},\prod_p g_p)\in O_A(\mathbb{A}_{\mathbb{Q}}).$ It is well-known that $C(A)$, the genus class of $A$,  is isomorphic to 
$
O_A(\mathbb{Q})\backslash O_A(\mathbb{A}_f)/O_A(\prod_p\mathbb{Z}_p).$ 
Suppose that  $\{L_i\subset \mathbb{Q}^n: 1 \leq i\leq h   \}$  is a representative set for the genus class of the lattice $(\mathbb{Z}^n,A)$ such that $(L_i,A)$ is isomorphic to $(\mathbb{Z}^m,A_i),$ which means $L_i^{\intercal}AL_i=A_i.$ 
We extend $K_{r,B}(\vec{X},\vec{Y})$ into a function on  $\mathcal{L}_{A,B}$ and denote the extension by $K_{r,B}\big((\vec{X},L_1),(\vec{Y},L_2)  \big)$  again. Define 
\[
K_{r,B} \big((\vec{X},L_1),  (\vec{Y},L_2) \big):=
\begin{cases} K_{r,B}(\vec{X},\vec{Y}) &\text{ if } L_1=L_2,
\\
0 &\text{ otherwise.}
\end{cases}
\]
We sum $K_{r,B}(\vec{X},\vec{Y})$ over the orbit of $O_A(\mathbb{Q})$, and obtain 
\[
\mathcal{K}_{r,B} \big((\vec{X},L_1),  (\vec{Y},L_2) \big):= \sum_{\gamma \in O_A(\mathbb{Q})} K_{r,B} \big(\gamma(\vec{X},L_1),  (\vec{X},L_2) \big).
\]
Note that $\mathcal{K}_{r,B}$ is invariant by the action of $O_A(\mathbb{Q})$ on the left:
\[
\mathcal{K}_{r,B}\big((\vec{X},L_1),  (\vec{Y},L_2) \big)=\mathcal{K}_{r,B} \big(\gamma(\vec{X},L_1),  (\vec{Y},L_2) \big)=\mathcal{K}_{r,B} \big((\vec{X},L_1), \gamma (\vec{Y},L_2) \big),
\]
for every $\gamma\in O_A(\mathbb{Q}).$ Hence, $\mathcal{K}_{r,B} \big((\vec{X},L_1),  (\vec{Y},L_2) \big)$ defines an automorphic kernel   on $O_A(\mathbb{Q})\backslash \mathcal{L}_{A,B} \times O_A(\mathbb{Q})\backslash \mathcal{L}_{A,B}.$ 
Given $(\vec{X},L_1),  (\vec{Y},L_2) \in \mathcal{L}_{A,B},$ it follows that 
\begin{equation}\label{eqdef}
\mathcal{K}_{r,B} \big((\vec{X},L_1),  (\vec{Y},L_2) \big)= \delta(L_1,L_2) \sum_{\gamma \in O_{L_2}}K_{r,B}(\gamma \gamma_{(1,2)}\vec{X},\vec{Y}),
\end{equation}
where \[
\delta(L_1,L_2)=\begin{cases}
1 \text{ if $L_1$ and $L_2$ are in the same genuss class of lattices, }
\\
0 \text{ otherwise,}
\end{cases}
\]
 and $O_{L_2}$ is the stabilizer of $L_2$ in the  orthogonal group $O_A(Q),$ and if $\delta(L_1,L_2)=1,$ then there exists  $ \gamma_{(1,2)} \in O_A(\mathbb{Q})$ which maps $L_1$  to $L_2$. Recall that $\vec{E}_B=\vec{E}\sqrt{B}.$ Fix
\(
\vec{\mathcal{E}_0}:=(\vec{E}_B, \mathbb{Z}^m)\in \mathcal{L}_{A,B}.
\) Let $O_{A,\vec{E}}(\mathbb{R})$ denote the stabilizer of $\vec{E}_B,$ then $O_{A,\vec{E}}(\mathbb{R})O_A(\prod_p \mathbb{Z}_p)$ is the stabilizer of $\mathcal{E}_0\in\mathcal{L}_{A,B},$ and we have the following isomorphism 
\[
\mathcal{L}_{A,B}=  O_A(\mathbb{A}_{\mathbb{Q}})/O_{A,\vec{\mathcal{E}_0}}(\mathbb{A}_{\mathbb{Q}}).
\] 
Therefore, we can view $\mathcal{K}_{r,B}$ as an automorphic point-pair invariant  function  on 
\[ O_A(\mathbb{Q})\backslash O_A(\mathbb{A}_{\mathbb{Q}})/O_{A,\vec{\mathcal{E}_0}}(\mathbb{A}_{\mathbb{Q}}) \times  O_A(\mathbb{Q})\backslash O_A(\mathbb{A}_{\mathbb{Q}})/ O_{A,\vec{\mathcal{E}_0}}(\mathbb{A}_{\mathbb{Q}}).\]

\subsubsection{Integration formula for $\va(B,r)$}
%
%
%
%
%
%
%
%
Let 
\[
\mathcal{S}_B:=\{(\vec{X},L)\in \mathcal{L}_{A,B}:  \vec{x}_j \in L, \text{ where }   \vec{x}_j \text{ is the $j$-th column of } \vec{X} \text{ for }1\leq j\leq n \}.
\]
Note that 
\(\mathcal{S}_B=\cup_{i=1}^{h}  \mathcal{S}_{L_i,B}, \)
where
\begin{equation}\label{lattice}  \mathcal{S}_{L_i,B}:=\{(\vec{X},L)\in \mathcal{S}_B: (L,A) \text{ is equivalent to } (L_i,A)\}. \end{equation}
 Note that $ \mathcal{S}_{L_i,B}$ is invariant by the action of  $O_A(\mathbb{Q}).$   Finally, we define the adelic  variance:
\[
\mathcal{VAR}(B,r):=\int_{O_A(\mathbb{Q})\backslash \mathcal{L}_{A,B}} \Big(\sum_{\mathcal{Q}\in O_A(\mathbb{Q})\backslash \mathcal{S}_B} \dfrac{\mathcal{K}_{r,B}(\mathcal{X},\mathcal{Q})}{|O_{\mathcal{Q}}(\mathbb{Q})|}- \mathcal{R}(B) \Big)^2 d\tilde{\mu}(\mathcal{X}),
\]
where $|O_{\mathcal{Q}}(\mathbb{Q})|$ is the size of the stabilizer of a representative of $\mathcal{Q}\in O_A(\mathbb{Q})\backslash \mathcal{S}_B,$ and $ d\tilde{\mu}(\mathcal{X})$ is a normalized $O(\mathbb{A}_Q)$ invariant measure such that 
\[
\int_{O_A(\mathbb{Q})\backslash \mathcal{L}_{A,B}} d\tilde{\mu}(\mathcal{X})=1,
\]
and
\begin{equation}\label{R(B)}
 \mathcal{R}(B)=\int_{O_A(\mathbb{Q})\backslash \mathcal{L}_{A,B}} \sum_{\mathcal{Q}\in O_A(\mathbb{Q})\backslash \mathcal{S}_B} \dfrac{\mathcal{K}_{r,B}(\mathcal{X},\mathcal{Q})}{|O_{\mathcal{Q}}(\mathbb{Q})|} d\tilde{\mu}(\mathcal{X}).
\end{equation}
%
%
\begin{prop}\label{lememl}
We have 
\begin{equation}
\begin{split}
\int_{O_A(\mathbb{Q})\backslash \mathcal{L}_{A,B}} \mathcal{K}_{r,B}(\mathcal{X},\mathcal{Y}) d\tilde{\mu}(\mathcal{X})&=\int_{O_A(\mathbb{Q})\backslash \mathcal{L}_{A,B}} \mathcal{K}_{r,B}(\mathcal{X},\mathcal{Y}) d\tilde{\mu}(\mathcal{Y})=1,
\\
 \mathcal{R}(B)&=R(B),
 \\
 \mathcal{VAR}(B,r)&=\va(B,r).
\end{split}
\end{equation}
\end{prop}
%
%
%
%
\begin{proof} We have  $O_A(\mathbb{Q})\backslash \mathcal{L}_{A,B}=\cup_{i=1}^{h} \{(L_i,\vec{X}): \vec{X}\in O_{L_i}\backslash V_{A,B}(\mathbb{R})\},$ where $O_{L_i}$ is the stabilizer of $L_i$ by the action of $O_A(\mathbb{Q}).$  Since the action of $O_A(\mathbb{A}_{\mathbb{Q}})$ is transitive on $ \mathcal{L}_{A,B}$ and $ d\mu$ is a Haar measure, it follows that 
\[
\tilde{\mu} \Big( \{(L_i,\vec{X}): \vec{X}\in  V_{A,B}(\mathbb{R})\} \Big)= \tilde{\mu} \Big( \{(L_j,\vec{X}): \vec{X}\in  V_{A,B}(\mathbb{R})\} \Big)
\]
for every $1\leq i,j \leq h.$
  Since \(\int_{O_A(\mathbb{Q})\backslash \mathcal{L}_{A,B}} d\tilde{\mu}(\mathcal{X})=1\), we have 
  \[
  \tilde{\mu}\Big(\{(L_i,\vec{X}): \vec{X}\in O_{L_i}\backslash V_{A,B}(\mathbb{R})\}  \Big)=\frac{\frac{1}{|O_{L_i}|}}{\Big(\sum_{L_i} \frac{1}{|O_{L_i}|}\Big) }.
  \]
  Recall that $\int_{V_{A_i,B}(\mathbb{R})} d\mu_i(\vec{Y})=1$ and $L_i^{\intercal}AL_i=A_i.$
This implies $ \Big(\sum_{L_i} \frac{1}{O_{L_i}}\Big)d\tilde{\mu}$ restricted to $\{(L_i,\vec{X}): \vec{X}\in O_{L_i}\backslash V_{A,B}(\mathbb{R})\} $  is equal to $d\mu$ on  $O_{L_i} \backslash V_{A,B}(\mathbb{R}).$
%
Let $\mathcal{Y}=(L,Y).$ By \eqref{eqdef}, we have  
\begin{equation*}
\begin{split}
\int_{O_A(\mathbb{Q})\backslash \mathcal{L}_{A,B}} \mathcal{K}_{r,B}(\mathcal{X},\mathcal{Y}) d\tilde{\mu}(\mathcal{X})&= \int_{O_{L}\backslash V_{A,B}(\mathbb{R})}  \sum_{\gamma \in O_{L}}K_{r,B}(\gamma\vec{X},\vec{Y})d\mu(\vec{X})
\\
& =\int_{ V_{A,B}(\mathbb{R})}  K_{r,B}(\vec{X},\vec{Y})d\mu(\vec{X})=1.
\end{split}
\end{equation*}
This completes the proof of the first identity. For the second identity, we have 
\begin{equation*}
\begin{split}
\Big(\sum_{A_i} \frac{1}{O_{L_i}}\Big) \mathcal{R}(B)= \sum_{i=1}^{h} \int_{O_{L_i}\backslash V_{A,B}(\mathbb{R})} \sum_{\mathcal{Q}\in O_A(\mathbb{Q})\backslash  \mathcal{S}_{L_i,B}} \dfrac{\mathcal{K}_{r,B}((L_i,\vec{X}),\mathcal{Q})}{|O_{\mathcal{Q}}(\mathbb{Q})|} d\mu(\vec{X}).
\end{split}
\end{equation*}
By unfolding $\sum_{\mathcal{Q}\in O_A(\mathbb{Q})\backslash  \mathcal{S}_{L_i,B}} \dfrac{\mathcal{K}_{r,B}((L_i,\vec{X}),\mathcal{Q})}{|O_{\mathcal{Q}}(\mathbb{Q})|}$, we have 
\begin{equation*}
\begin{split}
\Big(\sum_{A_i} \frac{1}{O_{L_i}}\Big) \mathcal{R}(B)&= \sum_{i=1}^{h} \frac{1}{|O_{A_i}(\mathbb{Z})|} \int_{ V_{A_i,B}(\mathbb{R})} \sum_{Q\in V_{A_i,B}(\mathbb{Z})} {K_{r,B}(\vec{X},Q)} d\mu_i(\vec{X})
\\
&= \sum_{i=1}^{h} \frac{1}{|O_{A_i}(\mathbb{Z})|}R_{A_i}(B)=\Big(\sum_{A_i} \frac{1}{|O_{A_i}(\mathbb{Z})|}\Big) R(B).
\end{split}
\end{equation*}
Finally,  by Lemma~\ref{lem22}, we have
\begin{equation*}
\begin{split}
\Big(\sum_{A_i} \frac{1}{|O_{A_i}(\mathbb{Z})|}\Big) \mathcal{VAR}(B,r)&= \sum_{i=1}^{h} \int_{O_{L_i}\backslash V_{A,B}(\mathbb{R})} \Big( \big(\sum_{\mathcal{Q}\in O_A(\mathbb{Q})\backslash  \mathcal{S}_{L_i,B}} \dfrac{\mathcal{K}_{r,B}(\mathcal{X},\mathcal{Q})}{|O_{\mathcal{Q}}(\mathbb{Q})|}\big)- R(B) \Big)^2 d\tilde{\mu}(\mathcal{X})
\\
&=\sum_{A_i} \int_{O_{A_i}(\mathbb{Z})\backslash V_{A_i,B}(\mathbb{R})} \Big(\big( \sum_{Y\in O_{A_i}(\mathbb{Z})\backslash  V_{A_i,B}(\mathbb{Z}) } \frac{\tilde{K}_{r}(\vec{X},\vec{Y})}{|O_{L_i,\vec{Y}(\mathbb{Z})}|}\big)- R(B)  \Big)^2  d
\mu(\vec{X})
\\
&=\Big(\sum_{A_i} \frac{1}{|O_{A_i}(\mathbb{Z})|}\Big)\va(B,r).
\end{split}
\end{equation*}
This completes the proof of the lemma.
\end{proof}
\subsection{Siegel variance in terms of the Weyl sums}\label{spec}
In this section, we write  the spectral decomposition of  $\mathcal{K}_{r,B}.$ Let $\{ \phi_{\pi,j}(\alpha)\}$ be an orthonormal basis of $L^2\Big(O_A(\mathbb{Q})\backslash O_A(\mathbb{A}_{\mathbb{Q}})/ O_{A,\vec{E}}(\mathbb{R})O_A(\prod_p \mathbb{Z}_p)\Big),$ where $\pi$ is an automorphic representation and $ \phi_{\pi,j}$ is an  $O_{A,\vec{E}}(\mathbb{R})O_A(\prod_p \mathbb{Z}_p)$ invariant vectors in $\pi.$  
We write $ \pi=\pi_{\infty}\prod_p \pi_p,$ where $\pi_p$ and $\pi_{\infty}$ are the local components of the automorphic representation $\pi.$ We identify $ O_A(\mathbb{Q})\backslash O_A(\mathbb{A}_{\mathbb{Q}})/O_{A,\vec{E}}(\mathbb{R})O_A(\prod_p \mathbb{Z}_p)=O_A(\mathbb{Q})\backslash \mathcal{L}_{A,B}.$ By Lemma~\ref{lememl}, we have 
$$
\mathcal{K}_{r,B}(\alpha,\beta)=1+\sum_{\pi} \sum_{j=1}^{d_{\pi}} h_{r}(\pi_{\infty}) \phi_{\pi,j}(\alpha)\bar{\phi}_{\pi,i}(\beta),
$$
where  the sum is over $ \phi_{\pi,j}$ such that
\[
\int_{O_A(\mathbb{Q})\backslash \mathcal{L}_{A,B}} \phi_{\pi,j}(\alpha)  d\tilde{\mu}(\alpha)=0,
\]
and $ h_{r}(\pi_{\infty})$ is the spherical transformation of the point-pair invariant kernel $K_{r,B},$ which is defined by:  
\begin{equation}\label{sphtrans}
h_{r}(\pi_{\infty})\phi_{\pi,j}(\alpha)=\int_{O_A(\mathbb{Q})\backslash \mathcal{L}_{A,B}} \mathcal{K}_{r,B}(\alpha,\beta)\phi_{\pi,j}(\beta) d\tilde{\mu}(\beta).
 \end{equation}
 By Lemma~\ref{lememl}, we obtain 
\[
 \mathcal{VAR}(B,r)=\int_{O_A(\mathbb{Q})\backslash \mathcal{L}_{A,B}}\Big(\sum_{\pi} \sum_{j=1}^{d_{\pi}} h_{r}(\pi_{\infty}) \phi_{\pi,j}(\alpha) W(\phi_{\pi,j},B)  \Big)^2 d\tilde{\mu}(\alpha),
\]
where 
 \begin{equation}\label{Wweyl}
W(\phi_{\pi,j},B):=\sum_{\mathcal{Q}\in O_A(\mathbb{Q})\backslash \mathcal{S}_B} \dfrac{ \phi_{\pi,j}(\mathcal{Q})}{|O_{\mathcal{Q}}(\mathbb{Q})|},
\end{equation}
which is a generalization of Weyl's sum associated to $\phi_{\pi,j}.$ By using the orthogonality of $\phi_{\pi,j}$, only the diagonal terms contribute to $ \mathcal{VAR}(B,r)$, and  we have the following proposition. 
\begin{prop}\label{Siegelvarp} We have 
\begin{equation}\label{Siegelvar}
 \mathcal{VAR}(B,r)=\sum_{\pi} \sum_{j=1}^{d_{\pi}} |h_{r}(\pi_{\infty})|^2 |W(\phi_{\pi,j},B)|^2.
\end{equation}
\end{prop}

\section{Harmonic polynomials}\label{sphesec}

\subsection{Harmonic polynomials for n=1}
In this section, we restrict ourself to the case $n=1$ and  cite some standard results on the spherical harmonic polynomials.  Let $F(\vec{x}):= \vec{x}^{\intercal} A\vec{x},$ where 
$\vec{x}=\begin{bmatrix}x_1\\ \vdots \\x_m \end{bmatrix}.$ Let  $A^{-1}=[a^{ij}]$ denote the inverse of $A,$ and 
\(
\Delta_A:=\sum_{i,j} a^{ij} \frac{\partial^2}{\partial x_i \partial x_j}\) be the Laplacian operator associated to $A.$
  Let $H_k$ be the space of harmonic polynomials of degree $k$  with respect to the symmetric matrix $A$, which is
  \[
  H_k:=\left\{p(\vec{x}):  \Delta_A p(\vec{x})=0, \text{ and }\deg{p}=k   \right\}.
  \]
  Let  $\vec{r}\in \mathbb{C}^m$ and $ \vec{r}^{\intercal}A \vec{r}=0.$   It is easy to check that
  \(
\Delta_A \langle \vec{x}, \vec{r} \rangle^k=0, 
\)  for $k=1$  the condition $ \vec{r}^{\intercal}A \vec{r}=0$  is not necessary.   It is well-known that $H_k$ is the span of polynomials of the form \(
\langle \vec{x}, \vec{r} \rangle^k.
\)  Moreover $H_k$ is invariant under the action of  $O_A(\mathbb{C})$ and form an irreducible representation of this group.

Let $V_N(\mathbb{R})=\{\vec{x}\in \mathbb{R}^m: F(\vec{x})=N  \}$ and   $f\in H_k$. Then the restriction of $f $ to $V_N(\mathbb{R})\subset \mathbb{R}^m$ defines an embedding of $H_k$ into $L^2(V_N(\mathbb{R})).$ Next we give the spectral decomposition of $L^2(V_N(\mathbb{R}))$ in terms of the harmonic polynomials.
\begin{prop}\label{ppp}
We have 
$$L^2(V_N(\mathbb{R}))=\oplus_{k} H_k $$
\end{prop}
\begin{proof}
This proposition is standard; see~\cite[Section 2.5.12]{Vergne} for the proof. 
\end{proof}

Fix $\vec{e}\in V_N(\mathbb{R})$ and let  $O_{A,\vec{e}} \subset O_A(\mathbb{R})$ be the centralizer of $\vec{e}$. It follows that  there exists a unique  $p_{k,\vec{e}}(\vec{x})\in H_k,$ such that 
\begin{equation}
\begin{cases}
p_{k,\vec{e}}(\vec{x})= p_{k,\vec{e}}(g\vec{x}), \text{ for every } g\in O_{A,\vec{e}} \text{ and } \vec{x}\in V_N(\mathbb{R}),
\\
p_{k,\vec{e}}(\vec{e})=1.
\end{cases}  \end{equation}
The following mean value theorem is standard  for the harmonic polynomials. 
 \begin{lem}\label{spherical}
Let $p_{k,\vec{e}}(\vec{x})$ be as above and $q(\vec{x})\in H_k.$ We have 
\[
\begin{split}
\int_{O_{A,\vec{E}}}  q(g\vec{x}) d\mu(g)&= q(\vec{e}) p_{k,\vec{e}}(\vec{x}),
\\
p_{k,\vec{e}}(u\vec{e})&=p_{k,\vec{e}}(u^{-1}\vec{e}), \text{ where }u\in O_{F}(\mathbb{R}),
\\
\int_{V_N(\mathbb{R})} q(\vec{x}) \overline{ p_{k,\vec{e}}}(\vec{x}) d\mu(\vec{x})&= q(\vec{e})|p_{k,\vec{e}}|^2, \text{ where } |p_{k,\vec{e}}|^2=  \int_{V_N(\mathbb{R})}  |p_{k,\vec{e}}(\vec{x})|^2 d\mu(\vec{x}).
\end{split}
\]
 \end{lem}
 \begin{proof}
The proof is an easy consequence of the uniqueness of $p_{k,\vec{e}}(\vec{x}).$
 \end{proof}
%
\subsection{Harmonic polynomials for general $n$}\label{harmgen}
In this section, we record  a generalization of the results of the previous section  from the work of Kashiwara and Vergne~\cite{Kashiwara}.  We give an orthonormal basis consisting of the generalized harmonic polynomials for $L^2(V_{A,B}(\mathbb{R})).$ We will use the result of this section later in Section~\ref{kernel} and~\ref{weight} and describe  the weight of the Siegel modular forms which appears in formula~\eqref{mainform}.

  We begin by defining some notations. Let $W:=\mathbb{R}^{n}$ and $W^{*}$ be its dual vector space. We take the symplectic space $V:=W+W^{*}$ with the symplectic form \(B(x_1+f_1,x_2+f_2)=f_2(x_1)-f_1(x_2).   \) Then $W$ and $W^{*}$ are complementary Lagrangian subspaces in $(V,B).$ Let $E:=(\mathbb{R}^m,A)$ be the inner product space with respect to  the  symmetric  form $A.$ Let $\mathcal{P}$ be the vector space of all complex valued polynomials on $Hom(W,E)=M_{m\times n}[\mathbb{R}],$ which is isomorphic to the space of complex valued polynomials on $Hom(W^{\mathbb{C}},E^{\mathbb{C}})=M_{m\times n}[\mathbb{C}].$ We denote by $O_A(\mathbb{C})$ the orthogonal group of $A$ with complex coefficients. The group $GL(n,\mathbb{C})\times O_A(\mathbb{C}) $ acts on $\mathcal{P}$ via \((H,\sigma)P=P(\sigma^{-1}XH)   \), where $H\in GL_n(\mathbb{C})$ and $\sigma\in O_A(\mathbb{C}).$
  
  For $X\in Hom(W^{\mathbb{C}},E^{\mathbb{C}})$ consider the symmetric matrix $X^{\intercal} A X.$ The coefficients \((X^{\intercal}AX)_{i,j}\) for $1\leq i,j\leq n$ generate the algebra of all $O_A(\mathbb{C})$ invariant polynomials on $Hom(W,E).$ Thus we can describe the algebra $\mathcal{D}_A$ of all $O_A(\mathbb{C})$-invariant constant coefficient differential operators on $Hom(W,E)$ as follows. We fix a basis of $W^{\mathbb{C}}$ and an orthogonal basis of $E^{\mathbb{C}}.$ Writing $X$ in $Hom(W^{\mathbb{C}},E^{\mathbb{C}})$ as
  \(
  X=[x_{i,j}]_{m\times n}.
  \) The algebra  $\mathcal{D}_A$ is generated by the operators:
  \[
  \Delta_{i,j}=\sum_{l=1}^{m}\frac{\partial}{\partial x_{li}}\frac{\partial }{\partial x_{lj}},
  \]
  for any $1 \leq i,j\leq n.$
We define the space of the harmonic polynomials by 
\[
\mathcal{H}:=\big\{P\in \mathcal{P}: \text{ such that } \Delta_{i,j}P=0 \text{ for all }1 \leq i,j \leq n     \big\}.
\]
 $\mathcal{H}$ is stable  under the action of $GL(n,\mathbb{C})\times O_A(\mathbb{C}).$ We write $\mathcal{H}=\oplus \mathcal{H}(\lambda)$ for the decomposition of $\mathcal{H}$ in isotopic components under $ O_A(\mathbb{C}).$ We cite the following theorem from~ \cite[Theorem 2.5.41]{Vergne} 
\begin{thm}\label{weightdis}
The isotopic component $\mathcal{H}(\lambda)$ of $\mathcal{H}$ of type $\lambda$ under $O_A(\mathbb{C})$ is irreducible under $GL(n,\mathbb{C})\times O_A(\mathbb{C})$ and it is isomorphic to $\tau \otimes \lambda $ for some finite dimensional irreducible representation of $GL(n,\mathbb{C}).$  Moreover, the isotopic component of $\mathcal{H}(\tau)$ of $\mathcal{H}$ of type $\tau$ under $GL_n(\mathbb{C})$ is irreducible under $GL(n,\mathbb{C})\times O_A(\mathbb{C}).$ In other words the correspondence $\lambda\to \tau$ is injective. 
\end{thm}
%
%
 Let $f\in  \mathcal{H}(\lambda)$. Then the restriction of $f $ to $V_{A,B}(\mathbb{R})\subset Hom(W,E)=M_{m\times n}[\mathbb{R}]$ defines an embedding of $ \mathcal{H}(\lambda)$ into $L^2(V_N(\mathbb{R})).$
We  have the following generalization of Proposition~\ref{ppp}. 
\begin{thm}
We have
\[
L^2(V_{A,B}(\mathbb{R}))=\oplus \mathcal{H}(\lambda)
\]
\end{thm}
\begin{proof}
The space of all polynomial is dense in $L^2(V_{A,B}(\mathbb{R}))$. Let $Inv$ be the sub-algebra of  the $O_A(\mathbb{C})$ invariant polynomials.   The space of all polynomials is the direct sum of $\mathcal{P}=\mathcal{H}+\mathcal{H}Inv$; see \cite[Section 2.5.11]{Vergne}. Since the restriction of $Inv$ is constant on 
$V_{A,B}(\mathbb{R})$. Hence, $L^2(V_{A,B}(\mathbb{R}))=\mathcal{H}=\oplus \mathcal{H}(\lambda).$
\end{proof}
Let 
\(
\mathcal{V}_{\vec{E}_B}:\mathcal{P}\to \mathbb{C}
\) be the evaluation of the polynomials at $\vec{E}_B.$ 
There exists a unique  $p_{\lambda,\vec{E}_B}(\vec{X})\in \mathcal{H}(\lambda)$ that represent the restriction of $\mathcal{V}_{\vec{E}_B}$ to $\mathcal{H}(\lambda)$, which means for every $q(\vec{X})\in \mathcal{H}(\lambda)$, we have 
\begin{equation}\label{sphpoly}
\int_{V_{A,B}(\mathbb{R})} q(\vec{X}) \overline{ p_{\lambda,\vec{E}_B}}(\vec{X}) d\mu(\vec{X})= q(\vec{E}_B).
\end{equation}
\begin{lem}\label{invp}
We have 
\[
 p_{\lambda,\vec{E}_B}(\alpha \vec{X})= p_{\lambda,\vec{E}_B}( \vec{X})
\]
for every $\alpha \in O_{A,\vec{E}}(\mathbb{R}).$ Moreover, we have 
\[
 p_{\lambda,\vec{E}_B}(g^{-1} \vec{E}_B)= p_{\lambda,\vec{E}_B}( g\vec{E}_B)
\]
for every $g \in O_{A}(\mathbb{R}).$ Finally 
\[
 p_{\lambda,\vec{E}_B}= \tau(\sqrt{B})^{-1}p_{\lambda,\vec{E}}.
\]
\end{lem}
\begin{proof}
Note that the functional  $\mathcal{V}_{\vec{E}_B}$ is invariant by $O_{A,\vec{E}}(\mathbb{R}),$ which means
\[
\mathcal{V}_{\vec{E}_B} (q(\vec{X}))= \mathcal{V}_{\vec{E}_B}  (q(\alpha\vec{X}))
\]
for every $\alpha \in O_{A,\vec{E}}(\mathbb{R}).$ This concludes the first part of the lemma. Let $\mathcal{P}_{O_{A,\vec{E}}(\mathbb{R})}$ be the set of harmonic polynomials which are invariant by $O_{A,\vec{E}}(\mathbb{R}).$ There is an involution $\sigma$ defined on  $\mathcal{P}_{O_{A,\vec{E}}(\mathbb{R})}$ as follows. For $q\in \mathcal{P}_{O_{A,\vec{E}}(\mathbb{R})} $ 
and $\vec{X}=g\vec{E}_B$ define 
\[
\sigma(q)(\vec{X}):= q(g^{-1}\vec{E}_B).
\]
It is easy to see that $\sigma(q) \in \mathcal{P}_{O_{A,\vec{E}}(\mathbb{R})}$ and 
\(
\mathcal{V}_{\vec{E}_B} (q)= \mathcal{V}_{\vec{E}_B}  (\sigma(q)) \). This implies $\sigma(p_{\lambda,\vec{E}_B})=p_{\lambda,\vec{E}_B}$, which concludes the second part of the lemma. Finally, we have 
\begin{equation}
\begin{split}
\int_{V_{A,B}(\mathbb{R})} q(\vec{X}) \overline{ \tau(\sqrt{B})^{-1}p_{\lambda,\vec{E}}}(\vec{X}) d\mu(\vec{X})&= \int_{V_{A,B}(\mathbb{R})} q(\vec{X}) \overline{ p_{\lambda,\vec{E} }}(\vec{X}\sqrt{B}^{-1}) d\mu(\vec{X})
\\
&=\int_{V_{A,I}(\mathbb{R})} q(\vec{Y}\sqrt{B}) \overline{ p_{\lambda,\vec{E} }}(\vec{Y}) d\mu(\vec{Y})
\\
&= q(\vec{E}\sqrt{B})=q(\vec{E}_B).
\end{split}
\end{equation}
This concludes the proof of the lemma. 
\end{proof}

\subsection{The weight space with a functional}\label{weights}
Let $ \mathcal{H}(\lambda)^{*}$ be the dual vector space of $ \mathcal{H}(\lambda)$. $GL_n(\mathbb{C})$ acts on $ \mathcal{H}(\lambda)^{*}$ by  ${\tau^{\intercal}}^{-1}.$ Every  $f\in \mathcal{H}(\lambda)$ defines a functional $\langle f,   \mathcal{H}(\lambda)^{*}\rangle \to \mathbb{C}.$ 

%
%
%
%
%
\section{The oscillator representations and  Weyl's sums}\label{SWF}
In this section, we describe the Schr\"odinger Model of the oscillator representation. We use this model  to construct an explicit automorphic Siegel's theta kernel. Next, we define the theta transfer  $\Theta(\pi)$ of an automorphic representation $\pi$ of $O_A$.  We show that  $\Theta(\phi_{\pi})$ is a holomorphic Siegel modular form with values in the dual space of vectors of $\pi_{\infty}$ and describe explicitly  its weight and its level in terms of $\pi_{\infty}$ and $A_{m\times m}$. We also show that $\Theta(\pi)$ is an eigenfunction of the Hecke operators at the unramified places. Finally, we relate the Weyl sums $W(\phi_{\pi},B)$ to  $\langle \Theta(\pi,B),\phi_{\pi}\rangle,$ where $\Theta(\pi,B)$ is  the $B$-th Fourier coefficient of $\Theta(\pi).$ This generalizes the result of Shintani~\cite{Shintani}.

\subsection{The Schr\"odinger Model of the oscillator representation}
We begin by describing the oscillator representation. Let $W:=\mathbb{Q}^n$ and $W^*$ be its dual vector space.  Consider the $2n$ dimensional symplectic vector space $W\oplus W^*$with the symplectic  form:
\[
\langle (x_1,y_1),(x_2,y_2) \rangle:= y_2(x_1)-y_1(x_2).
\] 
We fix the lattices $L_{W}:=\mathbb{Z}^n \subset W$ and $L_{W^*}:=\mathbb{Z}^n\subset W^*.$
Let $E=\mathbb{Q}^m$ be an orthogonal vector space with the positive definite symmetric form 
\[
(\vec{x},\vec{y})= \vec{x}^{\intercal}A\vec{y}.
\]
We fix the lattice $L_E:=\mathbb{Z}^m\subset E$ and denote its dual lattice by $L_{E}^*:=A^{-1}\mathbb{Z}^m\subset E.$
Consider the $2mn$ dimensional  symplectic vector space 
 $ (W\oplus W^*)\otimes E$ with the symplectic form 
 \[
 \langle  w_1\otimes v_1 , w_2\otimes v_2  \rangle =\langle w_1,w_2 \rangle (v_1,v_2),
 \]
where $v_1,v_2\in E$ and $w_1,w_2 \in W\oplus W^*.$ Note that $L:=L_W\otimes L_E^* \oplus L_{W^*}\otimes L_E $ is a self dual lattice inside $ (W\oplus W^*)\otimes E.$
We write a complete polarization as $(W\oplus W^*)\otimes E = W\otimes E\oplus W^*  \otimes E ,$ which means
  $W \otimes E $ and $  W^*\otimes E$  are the isotropic subspace of the symplectic vector space $(W\oplus W^*)\otimes E.$ We consider the adelic points of  $(W\oplus W^*)\otimes E$ with respect to the self dual lattice $L.$ We identify $W^*\otimes E$ with $Hom(W,E).$ Let $\mathcal{S}\left(Hom(W,E)\otimes \mathbb{A}_{\mathbb{Q}}\right)$ be the  Schwartz-Bruhat functions defined on the adelic space $Hom(W,E)\otimes \mathbb{A}_{\mathbb{Q}}.$ 
  Fix $\psi$ to be the   continuous  additive character on $\mathbb{Q}\backslash A_\mathbb{Q}/\prod_p\mathbb{Z}_p$ which is defined as follows on a complete representative set:
\[
\psi\big((a_{\infty},0,0,\dots) \big):=\exp(2\pi i a_{\infty}).
\]
By using standard the standard basis in the lattices $L_{W\oplus W^*}$ and $L$, we identify the symplectic group $SP_{W\oplus W^*}(\mathbb{A}_\mathbb{Q})$ with  $SP_{2n}(\mathbb{A}_\mathbb{Q})$ and $SP_{(W\oplus W^*)\otimes E}(\mathbb{A}_\mathbb{Q})$ with  $SP_{2mn}(\mathbb{A}_\mathbb{Q}).$ We note that under these coordinates the matrix representation of $s\otimes I_{m\times m}\in SP_{2mn}$  for $s=\begin{bmatrix} g_{11}& g_{12}\\ g_{21}&g_{22}  \end{bmatrix}\in SP_{2n}(\mathbb{A}_\mathbb{Q})$ is 
\begin{equation}\label{emb}
\begin{bmatrix} g_{11}\otimes I_{m\times m}& g_{12}\otimes A\\ g_{21}\otimes A^{-1}&g_{22}\otimes I_{m\times m} \end{bmatrix}.
\end{equation}

 Weil defined the Metaplictic group $\widetilde{SP}_{2mn} [ \mathbb{A}_{\mathbb{Q}}]$ (double cover of the symplectic group $SP_{2mn}[ \mathbb{A}_{\mathbb{Q}}]$) and constructed the unitary oscillator representation $\omega_{\psi}.$  In what follows, we record some  properties of $\omega_{\psi}$  from  \cite[Section 2]{Howe}; we refer the reader to  \cite[Section 2]{Howe} and \cite{Gelbart} for the definition and further properties of  $\omega_{\psi}$. In the Schr\"odinger Model of the oscillator representation,  $\omega_{\psi}$ acts on  
 $L^2\left(Hom(W,E)\otimes \mathbb{A}_{\mathbb{Q}}\right).$ It is not convenient and necessary for our purpose to give the action of $\widetilde{SP}_{2mn} [ \mathbb{A}_{\mathbb{Q}}].$ We only need the action of a parabolic subgroup of $\widetilde{SP}_{2mn} [  \mathbb{A}_{\mathbb{Q}}],$ which we describe next.  Let $P\subset SP_{2mn}$ be the stabilizer of  $W\otimes E.$ Let $M\subset P$ (maximal levi subgroup) be the stabilizer of  $W \otimes E $ and $  W^*\otimes E$ and $N\subset P$ (maximal unipotent subgroup) be the subgroup which acts as identity on $W\otimes E.$ We have a factorization $P=MN.$  
    More concretely,
 \begin{equation}\label{subgroups}
 \begin{split}
 P&=\left\{\begin{bmatrix}(g^t)^{-1}  & *\\ 0&g  \end{bmatrix}: g \in GL(W^*\otimes E)  \right\},
 \\
M&=\left\{\begin{bmatrix}(g^t)^{-1}  & 0\\ 0&  g\end{bmatrix}: g \in GL(W^*\otimes E)  \right\} ,
\\
N&=\left\{\begin{bmatrix}I_m & n\\ 0&I_m  \end{bmatrix}: n: W^* \otimes E \to W\otimes E\  \text{ and }  n=n^{\intercal} \right\}.
\end{split}
 \end{equation}
 We denote the inverse image of $P$ and $M$  in $\widetilde{SP}_{2mn} [  \mathbb{A}_{\mathbb{Q}}]$ by $\widetilde{P}$ and $\widetilde{M}.$ It follows that $N$ has a unique lift in  $\widetilde{SP}_{2mn} [  \mathbb{A}_{\mathbb{Q}}]$, so we may regard $N\subset \widetilde{SP}_{2mn} [  \mathbb{A}_{\mathbb{Q}}].$ Given $\tilde{g}\in \widetilde{M}$, its image in $M$ will be denoted by of $g.$
 The oscillator representation acts as follows in the Schr\"odinger Model; see~\cite{Howe, Gelbart}. For $ \Phi \in L^2(W^*\otimes E)$ we have 
 \begin{equation}\label{schr}
 \begin{split}
  \omega_{\psi}\left( \begin{bmatrix}I_m & n\\ 0&I_m  \end{bmatrix} \right)\Phi(X)=\psi\left(\frac{1}{2} \langle  X , n(X)  \rangle\right)  \Phi(X),
\\
 \omega_{\psi}\left(\tilde{g} \right) \Phi (X) =\gamma(\tilde{g}) |\det (g)|^{-1/2} \Phi(g^{-1}(X)), 
 \end{split}
 \end{equation}
 where  $X\in W^*\otimes E, $ $g\in GL(W^*\otimes E)$ and $\gamma(\tilde{g})$  is a certain root of unity, and det is the usual determinant function on
$GL(W^*\otimes E)$, and $|.|$ denotes the standard absolute value on $A_{\mathbb{Q}}$. In particular, for  $(\alpha,\tilde{s})\in O_A\times \widetilde{GL}(W^*)\subset GL(W^*\otimes E),$ we have  
 \begin{equation}\label{schrod}
  \omega_{\psi}( (\tilde{\alpha},\tilde{s}))   \Phi (X) =\gamma(\tilde{s}) |\det (s)|^{-m/2} \Phi(\alpha^{-1}\circ X\circ {s^{\intercal}}^{-1}),  
 \end{equation}
 where $s^{\intercal}\in GL(W)$ is the transpose of $s$ and $\alpha^{-1}\circ X\circ {s^{\intercal}}^{-1} \in Hom(W,E)$ is the composition of the linear maps. Here we have for convenience replaced $\widetilde{O_A}$ with $O_A$ itself and identified $\widetilde{SP}_{2mn} [  \mathbb{A}_{\mathbb{Q}}] $ with the image of $ \omega_{\psi}$; see \cite[Section 4]{Howe} for further discussion.  

 \subsection{Construction of the Siegel theta kernel}\label{kernel} In this section, we construct Siegel's theta kernel.  We begin by defining the Siegel upper half place associated to the symplectic space  $(W\oplus W^*,\langle, \rangle).$ Let 
\[
\mathbb{D}:=\left\{Z\in  W^*\otimes{\mathbb{C}} \to W \otimes {\mathbb{C}} \text {  such that  } Z^{\intercal}=Z, \text{ and } \Im(Z) >0   \right\},
\]
where $\Im(Z)$ is obtained by taking the imaginary part of every matrix entry of $Z$ and $\Im(Z)>0$ means  $\Im(Z)$ is positive definite. Let $s:=\begin{bmatrix} g_{11}& g_{12}\\ g_{21}&g_{22}  \end{bmatrix}\in SP_{2n}(\mathbb{Q}),$ where $g_{11}\in Hom (W,W),$ $g_{12}\in Hom(W^*,W),$ $g_{21}\in Hom(W,W^*)$ and $g_{22}\in Hom(W^*,W^*).$ Then $s$ acts on $\mathbb{D}$ as follows: 
\begin{equation}\label{Mob}
s:Z\to (g_{11}\circ Z+g_{12})  \circ  (g_{21}\circ Z+g_{22})^{-1}
\end{equation}
where $\circ$ is the composition of linear maps in $Hom$. For $Z\in \mathbb{D}$ and $ f\in\mathcal{H}(\lambda)=\lambda\otimes\tau,$ where $\mathcal{H}(\lambda)$ is the irreducible representation of $GL(n,\mathbb{C})\times O_A(\mathbb{C})$ defined in Theorem~\ref{weightdis}, we define $\varphi_{f,Z}\in \mathcal{S}\left( Hom(W,E)(\mathbb{A}_{\mathbb{Q}}) \right)$,  as follows:  
 \[
 \varphi_{f,Z}(\vec{X}_{\infty},\prod_p{\vec{X}_p}):=
 \exp\left(i\pi tr(Z\vec{X}_{\infty}^{\intercal}A\vec{X}_{\infty}) \right)f(\vec{X}_{\infty})\prod_{p} 1_{\mathbb{Z}_p}(\vec{X}_p)  ,
 \]
 where $1_{\mathbb{Z}_p}(\vec{X}_p)=1$ if $\vec{X}_p\in \mathbb{Z}_p^d$ and $1_{\mathbb{Z}_p}(\vec{X}_p)=0$ otherwise. Let $\mathcal{H}(\lambda)^*$ be the dual vector space of $\mathcal{H}(\lambda)$ that is defined in Section~\ref{weights}. We  define $ \phi_{\lambda,Z} \in \mathcal{S}\left( Hom(W,E)(\mathbb{A}_{\mathbb{Q}}), \mathcal{H}(\lambda)^* \right) $ to be the unique function that satisfies: \( \langle  \phi_{\lambda,Z}, g\rangle = \varphi_{g,Z}\) for every $g\in  \mathcal{H}(\lambda).$
 Next, we describe  the automorphic properties of $\varphi_{f,Z}$ and $ \phi_{\lambda,Z}$  as a function of $Z$ on the Siegel half plane. Recall that by Theorem~\ref{weightdis}, 
 $$\tau\left(((g_{21}Z+g_{22})^{\intercal})^{-1}\right)  f(\vec{X_{\infty}})=f\left(\vec{X_{\infty}} \circ ((g_{21}Z+g_{22})^{\intercal})^{-1} \right).$$
 \begin{lem}\label{sigwe}
Let $f\in  \mathcal{H}(\lambda)$ and $\tilde{s}\in \widetilde{SP}_{2n}(\mathbb{R})$ where  $s=\begin{bmatrix} g_{11}& g_{12}\\ g_{21}&g_{22}  \end{bmatrix}\in SP_{2n}(\mathbb{R}).$ We have 
\begin{equation}
\omega_{\psi}(\tilde{s})  \varphi_{f,Z}= \gamma(\tilde{s}) \det (g_{21}Z+g_{22})^{-m/2}  \varphi_{\left( \tau(g_{21}Z+g_{22})^{\intercal} \right)^{-1}f,s(Z)},
\end{equation}
and equivalently,
\begin{equation}
\omega_{\psi}(\tilde{s})  \phi_{\lambda,Z}= \gamma(\tilde{s}) \det (g_{21}Z+g_{22})^{-m/2} \left( \tau(g_{21}Z+g_{22})^{\intercal} \right)^{-1} \phi_{\lambda,s(Z)},
\end{equation}

 \end{lem}
 \begin{proof}
 We refer the reader to \cite[Section 2.5.42]{Vergne}.
 \end{proof}

 Let $\Theta$ be the following distribution on  $\mathcal{S}\left( Hom(W,E)\right)$ which sends a function  to the sum of its values on the rational points of $Hom(W,E)(\mathbb{Q})$:
 \[
 \Theta(f):= \sum_{a\in Hom(W,E)(\mathbb{Q})} f(a).
 \]
 Let $ \vartheta(g,f):= \Theta(\omega_{\psi}(g)f).$ It is well-known that $SP_{2mn}(\mathbb{Q})$ splits in $\widetilde{SP}_{2mn}(\mathbb{A}_{\mathbb{Q}})$ and we consider $SP_{2mn}(\mathbb{Q})\subset \widetilde{SP}_{2mn}(\mathbb{A}_{\mathbb{Q}}).$  It follows (by a generalized poisson formula) that $ \vartheta(g,f)$ is invariant by the action of $SP_{2mn}(\mathbb{Q})$ on the left and it defines an automorphic function on $L^2\left(SP_{2mn}(\mathbb{Q})) \backslash \widetilde{SP}_{2mn}(\mathbb{A}_{\mathbb{Q}})\right).$

 For $\alpha\in O_A(\mathbb{A}_{\mathbb{Q}}) $, $s\in \widetilde{SP}_{2n}(\mathbb{\mathbb{A}_{\mathbb{Q}}})$ and  $f\in \mathcal{H}(\lambda)$ for some $\lambda$,    we define the Siegel theta kernel  $\vartheta(\alpha,\tilde{s},f,Z)$ to be the following:
 \begin{equation}\label{Siegel}
 \begin{split}
\vartheta(\alpha,\tilde{s},f,Z):=\Theta (\omega_{\psi}((\alpha,\tilde{s}))\varphi_{f,Z}).
\end{split}
 \end{equation}
Note that $\vartheta(\alpha,\tilde{s},f,Z)$ is $O_A(\mathbb{Q})\times SP_{2n}(\mathbb{Q})$ invariant, and it defines a kernel which transfers the space cusp forms $\mathcal{A}^0\big(O_A(\mathbb{Q})\backslash O_{F}(\mathbb{A}_{\mathbb{Q}}) \big)$ to the automorphic forms of 
$L^2\big(SP_{2n}(\mathbb{Q})\backslash \widetilde{SP}_{2n}(\mathbb{A}_{\mathbb{Q}})\big)$ (possibly zero) and vice versa. Similarly, we define $\theta(\alpha,\tilde{s},\lambda,Z)$ with values in $\mathcal{H}(\lambda)^*$ to be the unique function which satisfies 
\[
\langle \theta(\alpha,\tilde{s},\lambda,Z),g  \rangle =  \vartheta(\alpha,\tilde{s},g,Z)
\]
for every $g\in \mathcal{H}(\lambda).$
%
\subsection{The weight and the level of the theta lift} For $\alpha\in O_A(\mathbb{Q})\backslash O_{F}(\mathbb{A}_{\mathbb{Q}})  $ we write  
\(
\theta(\alpha,\lambda,Z):=\theta(\alpha,\tilde{I}_{n\times n},\lambda,Z), \) where $\tilde{I}_{n\times n}$ is the identity element of  $\widetilde{SP}_{2n}(\mathbb{A}_{\mathbb{Q}})$. In this section, we show that $\theta(\alpha,\lambda,Z)$ is a holomorphic Siegel modular form of $Z$ with values in the vector space $\mathcal{H}(\lambda)^*$. Moreover, we show that its weight is given by the irreducible representation $\gamma\det^{m/2}(\tau^{\intercal})^{-1}$ and its level by the level of  $A$. We begin by defining the associated congruence subgroup of $SP_{2n}(\mathbb{Z}).$ Let $D$ be the level of $A$ which is the  smallest integer such that $DA^{-1}$ is integral and has even entries on its diagonal. We define the congruence subgroup $ \Gamma_0^n(D)\subset SP_n(\mathbb{Z}):$
\begin{equation}\label{congs}
 \Gamma_0^n(D):=\Big\{\begin{bmatrix}g_{11}& g_{12}\\ g_{21}&g_{22} \end{bmatrix} \in SP_n(\mathbb{Z}): g_{21}\in DM_{n\times n}(\mathbb{Z}), \text{ and } g_{12}\in 2 M_{n\times n}(\mathbb{Z})  \Big\}.
\end{equation}
\begin{prop} \label{ppww}Let $s_0=\begin{bmatrix}g_{11}& g_{12}\\ g_{21}&g_{22} \end{bmatrix}\in  \Gamma_0^n(D)$.
We have 
\begin{equation}
\theta(\alpha,\lambda,Z)= \gamma(s_0) \det (g_{21}Z+g_{22})^{-m/2}   ({\tau(g_{21}Z+g_{22})^{\intercal}})^{-1} \theta\left(\alpha,\lambda,s_0Z  \right).
\end{equation}
\end{prop}
\begin{proof}
It is enough to show that for every $g\in \mathcal{H}(\lambda),$ we have 
\[
\vartheta(\alpha,\tilde{I}_{n\times n},g,Z)= \gamma(\tilde{s}) \det (g_{21}Z+g_{22})^{-m/2}   \vartheta\left(\alpha,\tilde{I}_{n\times n},({\tau(g_{21}Z+g_{22})^{\intercal}})^{-1} g, s_0Z    \right).
\]
Since $SP_{2n}$ and $O_A$ commute in $SP_{2mn}$,   $s_0 \in  \Gamma_0^n(D) \subset SP_{2n}(\mathbb{Q})$ and $\Theta$ is invariant by $\widetilde{SP}_{2mn}(\mathbb{A}_{\mathbb{Q}}),$  we have 
\[
\vartheta(\alpha,\tilde{I}_{n\times n},g,Z)= \Theta (\omega_{\psi}(s_0)\circ \omega_{\psi}(\alpha)\varphi_{g,Z})=   \Theta ( \omega_{\psi}(\alpha) \circ \omega_{\psi}(s_0)\varphi_{g,Z}).
\]
By \eqref{emb}, the image of $s_0$ inside $SP_{2mn}$ is 
\[
\begin{bmatrix} g_{11}\otimes I_{m\times m}& g_{12}\otimes A\\ g_{21}\otimes A^{-1}&g_{22}\otimes I_{m\times m} \end{bmatrix}.
\]
We write $s_0=s_0^{\infty}s_{0,\infty},$ where $s_{0,\infty} \in \widetilde{SP}_{2n}\left(\prod_p SP_{2mn}(\mathbb{Q}_p)\right) $ and $s_{0,\infty} \in \widetilde{SP}_{2n}(\mathbb{R}).  $
By the definition of $D$, $\varphi_{g,Z}$ and \eqref{schr}, it follows that $\omega_{\psi}(s_0^{\infty})\varphi_{g,Z}=\varphi_{g,Z}.$  Finally by Lemma~\ref{sigwe}, we have  
\[
\omega_{\psi}(s_{0,\infty})\varphi_{g,Z}=  \gamma(\tilde{s}_0) \det (g_{21}Z+g_{22})^{-m/2}  \varphi_{\left( \tau(g_{21}Z+g_{22})^{\intercal} \right)^{-1}g,s(Z)}.
\]
This concludes the proof of our Proposition. 
\end{proof}
%
%
%
Proposition~\ref{ppww} implies that  $\theta(\alpha,\lambda,Z)$ has weight $(\tau_{\pi_{\infty}}^{\intercal})^{-1}$ and level $\Gamma_0^{n}(D)\subset SP_n(\mathbb{Z}),$ where $D$ is the discriminant of $A.$

%
%
%
%
%
%
%
\subsection{Hecke operators and the theta lift}\label{Heckes}
In this section, we briefly explain the Hecke algebra of the orthogonal group $O_A$ and its dual pair $\widetilde{SP_{2n}}$ at the unramified  primes. We cite a result of Howe~\cite[Theorem 7.1]{Howee}, that implies the theta transfer sends the  eigenfunction of the Hecke operators of $O_A$ to the eigenfunction of the Hecke operators of $SP_{2n}$. 

Let $p$ be a prime number where $\gcd(p,D)=1.$ Let $\tilde{J}_p$ be  the maximal compact subgroup of $\widetilde{SP}_{2mn} (\mathbb{Q}_p).$ It follows that $\tilde{J}_p$ splits and is isomorphic to $\tilde{J}_p=SP_{2mn} (\mathbb{Z}_p)\times \{\pm 1\}$; see \cite[Section 3]{Howee}. Let $K_p$ and $K^{\prime}_p$ be the maximal compact subgroups of  $O_A(\mathbb{Q}_p)$ and $\widetilde{SP}_{2n} (\mathbb{Q}_p).$ Up to conjugation, we can assume that $K_p$ and $K^{\prime}_p$ contained in $J_p.$  Let $C^{\infty}_c(O_A(\mathbb{Q}_p)//K_p)$ be the (Hecke) algebra of $K$-bi-invariant functions on $O_A(\mathbb{Q}_p)$. Define
$C^{\infty}_c(SP_{2n}(\mathbb{Q}_p)//K^{\prime}_p)$ similarly. Let $I(K_p, K^{\prime}_p)$ be the vectors fixed by $\omega_{\psi}(K_p)$ and by  $\omega_{\psi}(K^{\prime}_p).$
Then $\omega_{\psi}\left( C^{\infty}_c(O_A(\mathbb{Q}_p)//K_p)\right)$ and  $\omega_{\psi}\left( C^{\infty}_c(SP_{2n}(\mathbb{Q}_p)//K^{\prime}_p)\right)$ leaves $I(K_p, K^{\prime}_p)$ invariant.  We  consider the restrictions $\omega_{\psi}\left( C^{\infty}_c(O_A(\mathbb{Q}_p)//K_p)\right)| I(K_p, K^{\prime}_p).$  We cite the following result of Howe~\cite[Theorem 7.1]{Howee}\label{howheckek}
\begin{thm}[Howe] The restrictions $\omega_{\psi}\left( C^{\infty}_c(O_A(\mathbb{Q}_p)//K_p)\right)| I(K_p, K^{\prime}_p)$ and $\omega_{\psi}\left( C^{\infty}_c(SP_{2n}(\mathbb{Q}_p)//K^{\prime}_p)\right)| I(K_p, K^{\prime}_p)$ are
the same algebra of operators.
\end{thm}

Suppose that  $\phi_{\pi}$ is a smooth function which belongs to the automorphic irreducible representation of  $\pi$ of  $L^2\big(O_A(\mathbb{Q})\backslash O_{F}(\mathbb{A}_{\mathbb{Q}})\big)$ and is invariant by $K_p$ and $\pi_{\infty}=\lambda.$ We define 
\begin{equation}\label{irre}
\Theta(\phi_{\pi})(\lambda, Z):=\int_{O_A(\mathbb{Q})\backslash O_{A}(\mathbb{A}_{\mathbb{Q}})} \theta(\alpha,\lambda,Z)\overline{\phi_{\pi}}(\alpha) d\mu(\alpha).
\end{equation}
\begin{prop}\label{Heckeosc}
$\Theta(\phi_{\pi})(Z,f)$ is an eigenfunction of  $C^{\infty}_c(SP_{2n}(\mathbb{Q}_p)//K^{\prime}_p).$
\end{prop}
\begin{proof}
This is a consequence of of Theorem~\ref{howheckek}. See also Howe~\cite[Proposition  2.3]{Howe} for more details.
\end{proof}

%
%

 \subsection{Weyl's sums and the Fourier coefficient of the theta lift}\label{weight} Let $B\in Hom(W,W^{*})(\mathbb{Z})$ be a positive symmetric definite matrix $B^{\intercal}=B.$ Recall 
 \[
 N=\left\{\begin{bmatrix}I_m & n\\ 0&I_m  \end{bmatrix}\Big| n: W^* \otimes E \to W\otimes E\  \text{ and }  n=n^{\intercal} \right\} \subset SP_{2n}.
 \]  
 Note that by definition~\ref{congs},  $\Theta(\phi_{\pi,j})(\lambda, Z)$ is invariant by sending $Z$ to $Z+2n$ where $n\in N(\mathbb{Z}).$ 
 We define the $B$-th Fourier coefficient of $\Theta(\phi_{\pi})(Z,f)$ which is an element of $\mathcal{H}(\lambda)^*$ as follows:
  \begin{equation}\label{thetas}
 \Theta(\phi_{\pi,j})(\lambda, B):=\exp\left(-i\pi tr(ZB) \right) \int_{N(\mathbb{Q}) \backslash N(A_{\mathbb{Q}})} \Theta(\phi_{\pi})(\lambda, Z+2n) \psi\left(-tr(nB) \right) dn.
  \end{equation}
  Recall the Weyl sums $W(\phi_{\pi,j},B)$  and  $p_{\lambda,\vec{E}_B}\in \mathcal{H}(\lambda)$ defined in \eqref{Wweyl} and  \eqref{sphpoly} respectively.
 \begin{thm}\label{weylfourier}
 We have
 \[
\langle \Theta(\phi_{\pi,j})(\lambda, B), p_{\lambda,\vec{E}_B} \rangle =W(\phi_{\pi,j},B).
 \]
 \end{thm}
\begin{proof}
By Lemma~\ref{sigwe}, we have
\(
\varphi_{ f,Z+2n}= \omega_{\psi}(2n)   \varphi_{f,Z}.
\) By \eqref{schr}, we have 
  \[
   \varphi_{ f,Z+2n}=\omega_{\psi}(2n)  \varphi_{f,Z}=  \psi\left(\langle  \alpha^{-1}\vec{ H} , n(\alpha^{-1}\vec{ H})  \rangle \right)   \varphi_{f,Z} \]
  Therefore, 
 
\begin{gather*}
\begin{split}
 \langle \Theta(\phi_{\pi,j})(\lambda, B), p_{\lambda,\vec{E}_B} \rangle= \exp\left(-i\pi tr(ZB) \right)  \int_{N(\mathbb{Q}) \backslash N(A_{\mathbb{Q}})} \int_{O_A(\mathbb{Q})\backslash O_A(\mathbb{A}_{\mathbb{Q}})} \sum_{\vec{H}\in Hom(W,E)(\mathbb{Q})}
 \\
 \psi\left(\langle  \alpha^{-1}\vec{ H} , n(\alpha^{-1}\vec{ H})  \rangle -tr(nB)\right)   \varphi_{p_{\lambda,\vec{E}_B},Z}(\alpha^{-1}\vec{H}) \bar{\phi}_{\pi,i}(\alpha) dn d\alpha,
 \end{split}
\end{gather*}
where 
 \[
 \varphi_{p_{\lambda,\vec{E}_B},Z}(\vec{X}_{\infty},\prod_p{\vec{X}_p}):=
 \exp\left(i\pi tr(Z\vec{X}_{\infty}^{\intercal}A\vec{X}_{\infty}) \right)p_{\lambda,\vec{E}_B}(\vec{X}_{\infty})\prod_{p} 1_{\mathbb{Z}_p}(\vec{X}_p).
 \]
 Note that $\langle  \alpha^{-1}\vec{ H} , n(\alpha^{-1}\vec{ H})  \rangle=  tr(n \vec{ H}^{\intercal}({\alpha^{-1}})^{\intercal}A\alpha^{-1}\vec{H} )=  tr(n \vec{ H}^{\intercal}A\vec{H} ) $ is independent of $\alpha.$ Moreover, $ \varphi_{p_{\lambda,\vec{E}_B},Z}(\alpha^{-1}\vec{H}) = 0$ unless $\alpha^{-1}\vec{H} \in Hom(W,E)(\prod_p \mathbb{Z}_p)$, which implies $ \vec{ H}^{\intercal}A\vec{H} \in Hom(W,E)(\mathbb{Z}).$ By the orthogonality of the additive character $\psi,$ for $ \varphi_{p_{\lambda,\vec{E}_B},Z}(\alpha^{-1}\vec{H})\neq 0 $ we have 
 \[
  \int_{N(\mathbb{Q}) \backslash N(A_{\mathbb{Q}})} \psi\left(\langle  \alpha^{-1}\vec{ H} , n(\alpha^{-1}\vec{ H})  \rangle -tr(nB)\right) dn=\delta(\vec{ H}^{\intercal}A\vec{H}=B).
 \]
 where 
 \[
\delta(X,Y)=\begin{cases}
1 &\text{ if } X=Y
\\ 
0 &\text{ otherwise.}
\end{cases}
 \]
 Therefore, 
\begin{gather*}
\begin{split}
 \langle \Theta(\phi_{\pi,j})(\lambda, B), p_{\lambda,\vec{E}_B} \rangle=    \int_{O_A(\mathbb{Q})\backslash O_A(\mathbb{A}_{\mathbb{Q}})} \sum_{\vec{H}\in V_{A,B}(\mathbb{Q})}
  \varphi_{p_{\lambda,\vec{E}_B},Z}(\alpha^{-1}\vec{H}) \bar{\phi}_{\pi,i}(\alpha) \prod_{p} 1_{\mathbb{Z}_p}(\alpha^{-1}\vec{H}_p) d\alpha.
 \end{split}
\end{gather*}
 Recall that $\phi_{\pi,i}$ is defined on  $ O_A(\mathbb{Q})\backslash O_A(\mathbb{A}_{\mathbb{Q}})/O_{A,\vec{E}}(\mathbb{R})O_A(\prod_p \mathbb{Z}_p).$ By Lemma~\ref{invp}, $\varphi_{p_{\lambda,\vec{E}_B},Z}(\alpha^{-1}\vec{H}),$ as a function of $\alpha$, is also defined on $ O_A(\mathbb{A}_{\mathbb{Q}})/O_{A,\vec{E}}(\mathbb{R})O_A(\prod_p \mathbb{Z}_p).$  We identify  $ O_A(\mathbb{Q})\backslash O_A(\mathbb{A}_{\mathbb{Q}})/O_{A,\vec{E}}(\mathbb{R})O_A(\prod_p \mathbb{Z}_p)$ with $O_A(\mathbb{Q})\backslash \mathcal{L}_{A,B}.$ Recall from~\ref{adelicsec} $O_A(\mathbb{Q})\backslash \mathcal{L}_{A,B}=\cup_{i=1}^{h} \{(L_i,\vec{X}): \vec{X}\in O_{L_i}\backslash V_{A,B}(\mathbb{R})\}.$  Hence,  
  \begin{equation*}
 \begin{split}
 &\int_{O_A(\mathbb{Q})\backslash O_A(\mathbb{A}_{\mathbb{Q}})} \sum_{\vec{H}\in V_{A,B}(\mathbb{Q})} \varphi_{p_{\lambda,\vec{E}_B},Z}(\alpha^{-1}\vec{H}) \bar{\phi}_{\pi,i}(\alpha) \prod_{p} 1_{\mathbb{Z}_p}(\alpha^{-1}\vec{H}_p) d\alpha
 \\
 &= \sum_{L_i}\frac{1}{|O_{L_i}|} \sum_{\vec{H}\in V_{A,B}(\mathbb{Q})}\int_{O_A(\mathbb{R})/O_{A,\vec{E}}} 
 \phi_{\pi,j}((L_i,g \vec{E}))  p_{\lambda,\vec{E}_B}( g^{-1}\vec{H}) \prod_{p} 1_{\mathbb{Z}_p}(\alpha_i^{-1}\vec{H}_p) dg,
 \end{split}
 \end{equation*}
 where we have; see~\eqref{lattice}
 $$\prod_{p} 1_{\mathbb{Z}_p}(\alpha_i^{-1}\vec{H}_p)\begin{cases}
  1 &\text{ if }  H\in  \mathcal{S}_{L_i,B}
  \\
  0 &\text{ otherwise.}
  \end{cases}
$$
 Therefore,

 \begin{equation*}
 \begin{split}
 \langle \Theta(\phi_{\pi,j})(\lambda, B), p_{\lambda,\vec{E}_B} \rangle=   \sum_{L_i}\frac{ 1}{|O_{L_i}|} \sum_{\vec{H}\in  \mathcal{S}_{L_i,B}}\int_{O_A(\mathbb{R})/O_{A,\vec{E}}} 
 \phi_{\pi,j}((L_i,g \vec{E}))  p_{\lambda,\vec{E}_B}( g^{-1}\vec{H}) dg.
 \end{split}
 \end{equation*}
  Since $\vec{H}\in V_{A,B}(\mathbb{Q}),$ there exists  $\alpha_{\vec{H}}\in O_A(\mathbb{R})/O_{A,\vec{E}}$ such that  $\alpha_{\vec{H}}\vec{E}=\vec{H}.$
By the symmetry of $p_{\lambda,\vec{E}_B}$ proved in  Lemma~\ref{invp}, we have  
\begin{equation*}
\begin{split}
\int_{O_A(\mathbb{R})/O_{A,\vec{E}}} 
 \phi_{\pi,j}((L_i,g \vec{E}))  p_{\lambda,\vec{E}_B}( g^{-1}\vec{H}) dg&= \int_{O_A(\mathbb{R})/O_{A,\vec{E}}} 
 \phi_{\pi,j}((L_i,g \vec{E}))  p_{\lambda,\vec{E}_B}( g^{-1}\alpha_{\vec{H}}\vec{E}) dg
 \\
 &= \int_{O_A(\mathbb{R})/O_{A,\vec{E}}}  \phi_{\pi,j}((L_i,\alpha_{\vec{H}}g \vec{E}))  p_{\lambda,\vec{E}_B}( g^{-1}\vec{E}) dg
 \\
  &= \int_{O_A(\mathbb{R})/O_{A,\vec{E}}}  \phi_{\pi,j}((L_i,\alpha_{\vec{H}}g \vec{E}))  p_{\lambda,\vec{E}_B}( g\vec{E}) dg
  \\
 &= \phi_{\pi,j}((L_i,\alpha_{\vec{H}}\vec{E})) =\phi_{\pi,j}((L_i,\vec{H})).
\end{split}
\end{equation*}
Therefore,
\begin{equation*}
 \begin{split}
 \langle \Theta(\phi_{\pi,j})(\lambda, B), p_{\lambda,\vec{E}_B} \rangle&=   \sum_{L_i}\frac{1}{|O_{L_i}|} \sum_{\vec{H}\in  \mathcal{S}_{L_i,B}}\phi_{\pi,j}((L_i,\vec{H}))
 =W(\phi_{\pi,j},B).
 \end{split}
 \end{equation*}
This concludes the proof of our theorem. 
\end{proof}

\section{Proof of Theorem~\ref{Siegelthmm}}\label{final}
\begin{proof}
By Propositions~\ref{lememl} and~\ref{Siegelvarp}, we have 
\[
\va(B,r)=\sum_{\pi} \sum_{j=1}^{d_{\pi}} |h_{r}(\pi_{\infty})|^2 |W(\phi_{\pi,j},B)|^2.
\]
By Theorem~\ref{weylfourier} and Lemma~\ref{invp}, we have 
\begin{equation}\label{plam}
\begin{split}
W(\phi_{\pi,j},B)= \langle \Theta(\phi_{\pi,j})(\lambda, B), p_{\lambda,\vec{E}_B} \rangle&=  \langle \Theta(\phi_{\pi,j})(\lambda, B), \tau(\sqrt{B})^{-1}p_{\lambda,\vec{E}} \rangle 
\\
&=\langle  {\tau(\sqrt{B})^{\intercal}}^{-1} \Theta(\phi_{\pi,j})(\lambda, B),p_{\lambda,\vec{E}} \rangle. 
\end{split}
\end{equation}
This concludes the proof of Theorem~\ref{Siegelthmm}. 
\end{proof}

\bibliographystyle{alpha}
\bibliography{revised}

\end{document}